\documentclass[a4paper, 12pt, draft, reqno]{amsart}
\usepackage[final]{graphicx}
\usepackage{amsmath}
\usepackage{amssymb}
\usepackage{latexsym}
\usepackage[cmtip,matrix,arrow]{xy}
\usepackage{rotating}
\UseComputerModernTips
\CompileMatrices

\newcommand{\ZZ}{{\mathbb Z}}
\newcommand{\ZZo}{{\mathbb Z}\backslash\{0\}}

\newcommand{\NN}{{\mathbb N}}
\newcommand{\CC}{{\mathbb C}}

\newcommand{\calB}{{\mathcal B}}

\DeclareMathOperator{\Hom}{Hom}
\DeclareMathOperator{\Ext}{Ext}

\DeclareMathOperator{\res}{res}
\DeclareMathOperator{\add}{add}
\DeclareMathOperator{\remo}{rem}
\DeclareMathOperator{\supp}{supp}

\DeclareMathOperator{\Mod}{-mod}

\newcommand{\wif}{\mbox{\rm if }}
\newcommand{\wand}{\mbox{\rm and }}

\DeclareMathOperator{\ind}{ind}

\DeclareMathOperator{\pr}{pr}

\newcommand{\too}{\longrightarrow}

\begin{document}
\theoremstyle{plain}
\newtheorem{thm}{Theorem}[section]
\newtheorem{prop}[thm]{Proposition}
\newtheorem{lem}[thm]{Lemma}
\newtheorem{cor}[thm]{Corollary}
\newtheorem{conj}[thm]{Conjecture}
\newtheorem{qn}[thm]{Question}
\newtheorem{claim}[thm]{Claim}
\newtheorem{defn}[thm]{Definition}
\theoremstyle{definition}
\newtheorem{rem}[thm]{Remark}
\newtheorem{ass}[thm]{Assumption}
\newtheorem{example}[thm]{Example}

\setlength{\parskip}{1ex}

\title[Diagrammatic Kazhdan-Lusztig theory for the (walled) Brauer
algebra]{Diagrammatic Kazhdan-Lusztig theory for the (walled) Brauer algebra} 
\author{Anton Cox} 
\email{A.G.Cox@city.ac.uk}
\author{Maud De Visscher} 
\email{M.Devisscher@city.ac.uk}
\address{Centre for Mathematical Science\\ City
University\\ Northampton Square\\ London\\ EC1V 0HB\\ England.}
 \begin{abstract}
We determine the decomposition numbers for the Brauer and walled
Brauer algebra in characteristic zero in terms of certain polynomials
associated to cap and curl diagrams (recovering a result of Martin in
the Brauer case). We consider a second family of polynomials
associated to such diagrams, and use these to determine projective
resolutions of the standard modules. We then relate these two families
of polynomials to Kazhdan-Lusztig theory via the work of
Lascoux-Sch\"utzenberger and Boe, inspired by work of Brundan and
Stroppel in the cap diagram case.
 \end{abstract}

\maketitle

\section{Introduction}

Classical Schur-Weyl duality relates the representations of the
symmetric and general linear groups via their actions on tensor
space. The Brauer algebra was introduced in \cite{brauer} to play the
role of the symmetric group in a corresponding duality for the
symplectic and orthogonal groups. Over the complex numbers it is
generically semisimple \cite{brownbrauer}, indeed it can only be
non-semisimple if $\delta\in\ZZ$ \cite{wenzlbrauer}.

Building on work of Doran, Hanlon, and Wales \cite{dhw} we determined,
with Martin, the blocks of the Brauer algebra over $\CC$
\cite{cdm}. This block structure could be defined in terms of the
action of a Weyl group of type $D$ \cite{cdm2}, with a maximal
parabolic subgroup of type $A$ determining the dominant
weights.  The corresponding alcove geometry has associated translation
functors which can be used to provide Morita equivalences between
weights in the same facet \cite{cdm3}. More recently, Martin
\cite{marbrauer} has shown that the decomposition numbers for the
standard modules are given by the corresponding parabolic
Kazhdan-Lusztig polynomials.

The walled Brauer algebra was introduced in another generalisation of
Schur-Weyl duality, by changing the tensor space on which the
symmetric group acts. If instead a mixed tensor space (made of copies
of the natural module and its dual) is considered, then the walled
Brauer algebra plays the role of the symmetric group in the
duality. This was introduced independently by a number of authors
\cite{turwall, koikewall, bchlls}. In \cite{cddm} the walled Brauer
algebra was analysed in the same spirit as in \cite{cdm, cdm2}, and
the blocks were again described in terms of the action of a Weyl group
--- but this time of type $A$, with a maximal parabolic subgroup of
type $A\times A$ determining the dominant weights.

The Kazhdan-Lusztig polynomials associated to 
$(D_n,A_{n-1})$ and $(A_n,A_{r-1}\times A_{n-r})$ are two of the infinite
families associated with Hermitian symmetric spaces, and have already
been considered by a number of authors.  Lascoux and Sch\"utzenberger
\cite{lsgrass} considered the $(A_n,A_{r-1}\times A_{n-r})$ case
and gave an explicit formula for the coefficients in terms of certain
special valued graphs. This was extended to the other Hermitian
symmetric pairs by Boe \cite{boe}. A different combinatorial
description was given by Enright and Shelton \cite{es} in terms of an
associated root system.  (A more general situation has also been
considered by Brenti \cite{brenti} who describes the corresponding
polynomials in terms of shifted-Dyck partitions.)

The Brauer and walled Brauer algebras are examples of diagram
algebras. A quite different diagram algebra was introduced by Khovanov
\cite{khov1, khov2} in his work on categorifying the Jones
polynomial. Brundan and Stroppel have studied generalisations
of these algebras, relating them to a parabolic category
${\mathcal O}$ and the general linear supergroup \cite{bs1, bs2, bs3,
  bs4}. Along the way, Kazhdan-Lusztig polynomials of type
$(A_n,A_{r-1}\times A_{n-r})$ arise, and Brundan and Stroppel
re-express the combinatorial formalism of Lascoux and Sch\"utzenberger
in terms of certain cap diagrams.

In this paper we will determine the decomposition numbers for the
Brauer and walled Brauer algebras by analysing the blocks of these
algebras in the (combinatorial) spirit of Brundan and Stroppel. For
the Brauer algebra we introduce certain curl diagrams which
correspond to the graph formalism in Boe, while the walled Brauer
algebra involves only cap diagrams. The decomposition numbers for the
Brauer algebra were determined by Martin \cite{marbrauer}; our methods
give a uniform proof that includes the walled Brauer case.

One of the main organisational tools in our earlier work was the
notion of a tower of recollement \cite{cmpx}. We give a slight
extension of our earlier theory of translation functors for such
towers \cite{cdm3} and use this to reduce the decomposition number
problem to a combinatorial exercise. This is then solved using 
curl diagrams, thus giving a unified proof for the Brauer and
walled Brauer cases. 

In the Brauer case the combinatorial construction is related to that
given in \cite{marbrauer}. However, using cap and curl diagrams we are
able to explicitly calculate certain inverses to the decomposition
matrices for both Brauer and walled Brauer. The polynomial entries of
these matrices can be used to describe projective resolutions of the
standard modules in each case. (Again, this is in the spirit of
Brundan and Stroppel.)

We begin in Section \ref{basics} with a brief review of the basics of
Brauer and walled Brauer representation theory. Section \ref{tor}
reviews (and slightly extends) the tower of recollement formalism, and
the theory of translation functors in this context. Sections
\ref{capsec} and \ref{twicapsec} introduce two of the main
combinatorial constructions: oriented cap and curl diagrams. These are
used in Section \ref{repthry} to determine the decomposition numbers for our
algebras.

After providing a recursive formula for decomposition numbers in
Section \ref{recur} we define a second family of polynomials using
valued cap and curl diagrams in Section \ref{vcap}. These are used to
determine projective resolutions of standard modules in Section
\ref{proj}. Finally, the relation between the polynomials associated
to valued cap and curl diagrams and the construction of parabolic
Kazhdan-Lusztig polynomials by Lascoux-Sch\"utzenberger and Boe is
outlined in the Appendix.

We would like to thank Paul Martin for several useful discussions.

\section{The Brauer and walled Brauer algebras}\label{basics}

We will review some basic results about the representations of the
Brauer and the walled Brauer algebra. The two theories are very
similar; we will concentrate on the walled Brauer (which is less
familiar) and sketch the modifications required for the classical
Brauer algebra.  Details can be found in
\cite{cddm} for the walled Brauer algebra, and in \cite{cdm}
otherwise. We will restrict attention to the case where the ground
field is $\CC$, and assume that our defining parameter $\delta$ is
non-zero.

Let $n=r+s$ for some non-negative integers $r$, $s$. For $\delta\in
\CC$, the Brauer algebra $B_{n}(\delta)$ (which we will often denote
just by $B_{n}$) can be defined in terms of a basis of diagrams. We
will consider certain rectangles with $n$ marked nodes on each of the
northern and southern edges.  Brauer diagrams are then those
rectangles in which all nodes are connected to precisely one other by
a line.  Lines connecting nodes on the same edge are called arcs,
while those connecting nodes on opposite edges are called propagating
lines.  Multiplication of diagrams $A$ and $B$ is by concatenation, to
form a diagram $C$ which may contain some number ($t$ say) of closed
loops. To form a diagram in our basis we set $C$ equal to $\delta^tC'$
where $C'$ is the diagram obtained from $C$ by deleting all closed
loops.

Now decorate all Brauer diagrams in $B_n$ with a vertical wall
separating the first $r$ nodes on each edge from the final $s$ nodes
on each edge.  The walled Brauer algebra $B_{r,s}(\delta)$ (or just
$B_{r,s}$) is then the subalgebra of $B_n$ generated by those Brauer
diagrams in which arcs cross the wall, while propagating
lines do not.

For $\delta\neq 0$ let $e_{r,s}$ be $\delta^{-1}$ times the diagram
with all nodes connected vertically in pairs except for those adjacent
to the wall, which are connected across the wall. This is an
idempotent, and we have an algebra isomorphism 
$$B_{r-1,s-1}\cong e_{r,s}B_{r,s}e_{r,s}.$$
Via this isomorphism we have an exact localisation functor 
$$F_{r,s}:B_{r,s}\Mod\too B_{r-1,s-1}\Mod$$ taking a module $M$ to
$e_{r,s}M$, and a right exact globalisation functor $G_{r-1,s-1}$ in the
opposite direction taking a module $N$ to
$B_{r,s}e_{r,s}\otimes_{{e_{r,s}}B_{r,s}e_{r,s}}N$. There is a similar
idempotent $e_n\in B_n$ and algebra isomorphism $B_{n-2}\cong
e_nB_ne_n$, giving rising to corresponding localisation and
globalisation functors $F_n$ and $G_n$.

Let $\Sigma_r$ denote the symmetric group on $r$ symbols, and set
$\Sigma_{r,s}=\Sigma_r\times \Sigma_s$. There is an isomorphism
$$B_{r,s}/B_{r,s}e_{r,s}B_{r,s}\cong \CC\Sigma_{r,s}$$ and this latter
algebra has simple modules labelled by $\Lambda^{r,s}$, the set of
pairs of partitions of $r$ and $s$ respectively. By standard
properties of localisation it follows that if $r,s>0$ then the set of
simple modules for $B_{r,s}$ is labelled by
$$\Lambda_{r,s}=\Lambda^{r,s}\cup\Lambda_{r-1,s-1}.$$ As $B_{r,0}\cong
B_{0,r}\cong \Sigma_r$ we deduce that $\Lambda_{r,s}$ consists of all
pairs $\lambda=(\lambda^L,\lambda^R)$ such that $\lambda^L$ is a partition of
$r-t$ and $\lambda^R$ is a partition of $s-t$ for some $t\geq 0$. We
say that such a bipartition is of degree $\deg(\lambda)=(r-t,s-t)$, and
put a partial order on degrees by setting $(a,b)\leq(c,d)$ if $a\leq c$ and
$b\leq d$.

Let $\Lambda^n$ denote the set of partitions of $n$. Then by similar
arguments we see that the labelling set $\Lambda_n$ for simple
$B_n$-modules is given recursively by
$\Lambda_n=\Lambda^n\cup\Lambda_{n-2}$
and so $\Lambda_n$ consists of all partitions $\lambda$ of $n-2t$ for
some $t\geq 0$. We say that such a partition is of degree
$\deg(\lambda)=n-2t$. 

The $e_{r-t,s-t}$ with $0\leq t\leq \min(r,s)$ induce a heredity chain
in $B_{r,s}$, and so we can apply the theory of quasihereditary
algebras. In particular for each $\lambda\in \Lambda_{r,s}$ there is
an associated standard module $\Delta_{r,s}(\lambda)$ with simple head
$L_{r,s}(\lambda)$ and projective cover $P_{r,s}(\lambda)$. The
standard modules have an explicit description in terms of walled
Brauer diagrams and Specht modules for the various $\Sigma_{r-t,s-t}$,
and determining the decomposition numbers for these modules in terms
of their simple factors is equivalent to determining the simple
modules themselves.  In the same way the $B_n$ are quasihereditary,
with standard modules $\Delta_n(\lambda)$, with simple modules
$L_n(\lambda)$, and projective covers $P_n(\lambda)$.

By general properties of our heredity chain we have
$$G_{r,s}\Delta_{r,s}(\lambda)\cong \Delta_{r+1,s+1}(\lambda)$$
and
$$F_{r,s}\Delta_{r,s}(\lambda)\cong 
\left\{\begin{array}{ll}
\Delta_{r-1,s-1}(\lambda)&\wif \lambda\in\Lambda_{r-1,s-1}\\
0 &\wif\lambda\in\Lambda^{r,s}\end{array}\right.$$
We define a partial order on the set of all partitions (or all
bipartitions) by setting $\lambda\leq \mu$ if $\deg(\lambda)\leq
\deg(\mu)$. This is the opposite of the partial order induced by the
quasihereditary structure on $\Lambda_n$ or $\Lambda_{r,s}$.
Thus the decomposition multiplicity
$$[\Delta_{r,s}(\lambda):L_{r,s}(\mu)]$$ is zero unless $\lambda\leq
\mu$, and is independent of $(r,s)$ provided that
$\lambda,\mu\in\Lambda_{r,s}$ (and similarly for the Brauer case).

As our algebra is quasihereditary each projective module
$P_{r,s}(\lambda)$ has a filtration by standard modules. The
multiplicity of a given standard $\Delta_{r,s}(\mu)$ in such a
filtration is well-defined, and we denote it by
$$D_{\lambda\mu}=(P_{r,s}(\lambda):\Delta_{r,s}(\mu)).$$ By
Brauer-Humphreys reciprocity we have
$$D_{\lambda\mu}=[\Delta_{r,s}(\mu):L_{r,s}~(\lambda)]$$
(and hence $D_{\lambda\mu}$ is independent of $r$ and $s$). Again,
analogous results hold for the Brauer algebra, and we shall denote the
corresponding filtration multiplicities by $D_{\lambda\mu}$ also.

The algebra $B_{r,s}$ can be identified with a subalgebra of
$B_{r+1,s}$ (respectively of $B_{r,s+1}$) by inserting an extra
propagating line immediately to the left (respectively to the right)
of the wall. The corresponding restriction functors will be denoted
$\res^L_{r+1,s}$ and $\res^R_{r,s+1}$, with associated induction
functors $\ind^L_{r,s}$ and $\ind^R_{r,s}$. Similarly, $B_{n}$ is a
subalgebra of $B_{n+1}$ giving associated functors $\ind_n$ and $\res_{n+1}$

We will identify a partition with its associated Young diagram, and
let $\add(\lambda)$ (respectively $\remo(\lambda)$) denote the set of
boxes which can be added singly to (respectively removed singly from)
$\lambda$ such that the result is still a partition. Given such a box
$\epsilon$, we denote the associated partition by $\lambda+\epsilon$
(respectively $\lambda-\epsilon$). If we wish to emphasise that
$\epsilon$ lies in a given row ($i$ say) then we may denote it by
$\epsilon_i$.

By \cite[Theorem
  3.3]{cddm} we have

\begin{prop}\label{indres}
Suppose that $\lambda=(\lambda^L,\lambda^R)\in\Lambda^{r-t,s-t}$. If $t=0$
then
$$\res^L_{r,s}\Delta_{r,s}(\lambda^L,\lambda^R)\cong 
\bigoplus_{\epsilon\in\remo(\lambda^L)}\Delta_{r-1,s}(\lambda^L-\epsilon,\lambda^R).$$
If $t>0$ then there is a short exact sequence
$$0\too
\bigoplus_{\epsilon\in\remo(\lambda^L)}
\Delta_{r-1,s}(\lambda^L-\epsilon,\lambda^R)\too
\res^L_{r,s}\Delta_{r,s}(\lambda)\too
\bigoplus_{\epsilon\in\add(\lambda^R)}
\Delta_{r-1,s}(\lambda^L,\lambda^R+\epsilon)\too 0.$$
There is a similar result for $\res^R_{r,s}$ replacing $\remo(\lambda^L)$ by
$\remo(\lambda^R)$ and $\add(\lambda^R)$ by $\add(\lambda^L)$.
There is also a short exact sequence
$$0\too
\bigoplus_{\epsilon\in\remo(\lambda^L)}
\Delta_{r,s+1}(\lambda^L-\epsilon,\lambda^R)\too
\ind^R_{r,s}\Delta_{r,s}(\lambda)\too
\bigoplus_{\epsilon\in\add(\lambda^R)}
\Delta_{r,s+1}(\lambda^L,\lambda^R+\epsilon)\too 0$$
where the first sum equals $0$ if $\lambda^L=\emptyset$. Again there
is a similar result for $\ind^L_{r,s}$.
\end{prop}

There is an entirely analogous result for the Brauer algebra, where
the terms in the submodule of the restriction (or induction) of
$\Delta_n(\lambda)$ are labelled by all partitions obtained by
removing a box from $\lambda$, and those in the quotient module by all
partitions obtained by adding a box to $\lambda$. For example, we have
a short exact sequence
$$0\too
\bigoplus_{\epsilon\in\remo(\lambda)}
\Delta_{n+1}(\lambda-\epsilon)\too
\ind_n\Delta_n(\lambda)\too
\bigoplus_{\epsilon\in\add(\lambda)}
\Delta_{n+1}(\lambda+\epsilon)\too 0$$
where the first sum equals $0$ if $\lambda=\emptyset$.

It will be convenient to consider the Brauer and walled Brauer cases
simultaneously. In the walled Brauer case we will set $(a)=(r,s)$,
with $(a-1)=(r,s-1)$ and $(a+1)=(r+1,s)$. In the Brauer case we will
set $(a)=n$ with $(a-1)=n-1$ and $(a+1)=n+1$. Then $\Lambda_{(a)}$
will denote either $\Lambda_{r,s}$ or $\Lambda_n$ depending on the
algebra being considered, and similarly for $\Delta_{(a)}(\lambda)$
and any other objects or functors with subscripts.

\section{Translation functors}\label{tor}

In \cite{cdm3} we introduced the notion of translation functors for a
tower of recollement, and showed how they could be used to generate
Morita equivalence between different blocks. After a brief review of
this, we will show how this can be applied to the Brauer and walled
Brauer algebras. Details can be found in \cite[Section 4]{cdm3}.

Let $A_n$ with $n\in\NN$ form a tower of recollement, with associated
idempotents $e_n$ for $n\geq 2$. Let $\Lambda_n$ denote the set of
labels for the simple $A_n$ modules, which we call weights.  We denote
the associated simple, standard, and projective modules by
$L_n(\lambda)$, $\Delta_n(\lambda)$ and $P_n(\lambda)$
respectively. The algebra embedding arising from our tower structure
give rise to induction and restriction functors $\ind_n$ and
$\res_n$. For each standard module $\Delta_n(\lambda)$, the module
$\res_n\Delta_n(\lambda)$ has a filtration by standard modules with
well-defined multiplicities; we denote by $\supp_n(\lambda)$ the
multiset of labels for standard modules occurring in such a
filtration. We impose a crude order on weights by setting
$\lambda<\mu$ if there exists $n$ such that $\lambda\in\Lambda_n$ but
$\mu\notin\Lambda_n$. This is the opposite of the order induced by the
quasihereditary structure.

In such a tower we have isomorphisms $e_nA_ne_n\cong A_{n-2}$. Thus we
also have associated localisation functors $F_n$ and globalisation
functors $G_n$. Globalisation induces an embedding of $\Lambda_n$ inside
$\Lambda_{n+2}$, and an associated embedding of $\supp_n(\lambda)$
inside $\supp_{n+2}(\lambda)$, which becomes an identification if
$\lambda\in\Lambda_{n-2}$. We denote by $\supp(\lambda)$ the set
$\supp_n(\lambda)$ where $n>>0$.

Let ${\mathcal B}_n(\lambda)$ denote the set of weights labelling
simple modules in the same block for $A_n$ as $L_n(\lambda)$. Again
there is an induced embedding of ${\mathcal B}_n(\lambda)$ inside
${\mathcal B}_{n+2}(\lambda)$, and we denote by ${\mathcal
  B}(\lambda)$ the corresponding limit set. Given a weight $\lambda$,
we denote by $\pr_n^{\lambda}$ the functor which projects onto the
$A_n$-block containing $L_n(\lambda)$. We then define \emph{translation
functors} $\res_n^{\lambda}=\pr_{n-1}^{\lambda}\res_n$ and
$\ind_n^{\lambda}=\pr_{n+1}^{\lambda}\ind_n$.

We say that two weights $\lambda$ and $\lambda'$ are \emph{translation
equivalent} if
(i) we have
$${\mathcal B}(\lambda')\cap\supp(\lambda)=\{\lambda'\}
\quad\quad\wand\quad\quad
{\mathcal B}(\lambda)\cap\supp(\lambda')=\{\lambda\}$$
and (ii) for all $\mu\in{\mathcal B}(\lambda)$ there is a unique element
$\mu'\in{\mathcal B}(\lambda')\cap\supp(\mu)$ and 
$${\mathcal B}(\lambda)\cap\supp(\mu')=\{\mu\}.$$ 
When $\lambda$ and $\lambda'$ are translation equivalent we denote by 
$\theta:{\mathcal B}(\lambda)\too {\mathcal B}(\lambda')$ the
bijection taking $\mu$ to $\mu'$.

By \cite[Propositions 4.1 and 4.2]{cdm3} we have

\begin{thm}\label{equiv}
Suppose that $\lambda\in\Lambda_n$ and $\lambda'\in\Lambda_{n-1}$ are
translation equivalent, and that $\mu\in{\mathcal B}_n(\lambda)$ is
such that $\mu'\in{\mathcal B}_{n-1}(\lambda')$. \\
(i) We have
$$\res_n^{\lambda'}L_n(\mu)\cong L_{n-1}(\mu')\quad\quad
\ind_{n-1}^{\lambda}L_{n-1}(\mu')\cong L_{n}(\mu)$$
and
$$\ind_{n-1}^{\lambda}P_{n-1}(\mu')\cong P_n(\mu).$$
(ii) If $\tau\in{\mathcal B}_n(\lambda)$ is such that $\tau'\in{\mathcal
  B}_{n-1}(\lambda')$ then
$$[\Delta_n(\mu): L_n(\tau)]=[\Delta_{n-1}(\mu'):L_{n-1}(\tau')]$$
and
$$\Hom(\Delta_n(\mu),\Delta_n(\tau))\cong \Hom(\Delta_{n-1}(\mu'),
\Delta_{n-1}(\tau')).$$
(iii) If  $\mu\in {\mathcal B}_{n-2}(\lambda)$ then 
$$\res_{n}^{\lambda'}P_n(\mu)\cong P_{n-1}(\mu').$$
\end{thm}

The above result suggests that translation equivalent weights should
be in Morita equivalent blocks, but this is not true in general as
there will not be a bijection between the simple modules. However, by
a suitable truncation of the algebra we do get Morita equivalences. 

The algebra $A_n$ decomposes as 
$$A_n=\bigoplus_{\lambda\in\Lambda_n}P_n(\lambda)^{m_{n,\lambda}}$$
for some integers $m_{n,\lambda}$. Let
$1=\sum_{\lambda\in\Lambda_n}e_{n,\lambda}$ be the associated
orthogonal idempotent decomposition of the identity in $A_n$. There is
also a decomposition of $A_n$ into its block subalgebras
$$A_n=\bigoplus_{\lambda}A_n(\lambda)$$
where the sum runs over a set of block representatives. Now let
$\Gamma\subseteq {\mathcal B}_n(\lambda)$ and consider the idempotent
$e_{n,\Gamma}=\sum_{\gamma\in\Gamma}e_{n,\gamma}$. We define the
algebra $A_{n,\Gamma}(\lambda)$ by
$$A_{n,\Gamma}(\lambda)=e_{n,\Gamma}A(_n(\lambda)e_{n,\Gamma}.$$
By \cite[Theorem 4.5 and Corollary 4.7]{cdm3} we have

\begin{thm}\label{walls}
Suppose that $\lambda$ and $\lambda'$ are translation equivalent, with
$\lambda\in \Lambda_n$, and set 
$$\Gamma=\theta({\mathcal B}_n(\lambda))\subseteq{\mathcal
  B}_{n+1}(\lambda').$$
(i) The algebras $A_n(\lambda)$ and $A_{n+1,\Gamma}(\lambda')$ are Morita
equivalent. In particular, if $|{\mathcal B}_n(\lambda)|=
|{\mathcal B}_{n+1}(\lambda')|$ then $A_n(\lambda)$ and
$A_{n+1}(\lambda')$ are Morita equivalent.\\
(ii) For all $\mu\in{\mathcal B}_n(\lambda)$ we have
$$\Ext^i(\Delta_n(\lambda),\Delta_n(\mu))\cong
\Ext^i(\Delta_{n+1}(\lambda'),\Delta_{n+1}(\mu')).$$
\end{thm}

We will say that blocks ${\mathcal B}(\lambda)$ and ${\mathcal
  B}(\lambda')$ satisfying the condition in Theorem \ref{walls} are
\emph{weakly Morita equivalent}.

The notion of translation equivalent weights is motivated by the
translation principle in Lie theory, where translation functors give
equivalences for weights inside the same facet. Another common
situation in Lie theory involves the relationship between weights in a
pair of alcoves separated by a wall. There is also an analogue of this
in our setting.

We say that $\lambda'$ \emph{separates} $\lambda^-$ and $\lambda^+$ if
$${\mathcal B}(\lambda')\cap\supp(\lambda^-)=\{\lambda'\}=
{\mathcal B}(\lambda')\cap\supp(\lambda^+)$$
and 
$${\mathcal
  B}(\lambda^-)\cap\supp(\lambda')=\{\lambda^+,\lambda^-\}.$$
Whenever we consider a pair of weights $\lambda^-$ and $\lambda^+$
separated by $\lambda'$ we shall always assume that
$\lambda^-<\lambda^+$. By
\cite[Theorem 4.8]{cdm3} we have

\begin{thm}\label{local}
(i) If $\lambda'\in\Lambda_{n-1}$ separates $\lambda^-$ and $\lambda^+$
  then
$$\res_n^{\lambda'}L_n(\lambda^+)\cong L_{n-1}(\lambda').$$
(ii) If further we have
  $\Hom(\Delta_n(\lambda^+),\Delta_n(\lambda^-))\neq 0$ then
$$\res_n^{\lambda'}L_n(\lambda^-)=0$$
and $\ind_{n-1}^{\lambda^-}\Delta_n(\lambda')$ is a nonsplit
extension of $\Delta_n(\lambda^-)$ by $\Delta_n(\lambda^+)$ and has
simple head $L_n(\lambda^+)$.
\end{thm}

Suppose that $\lambda'$ and $\lambda^+$ are weights with
$\lambda'<\lambda^+$ such that for every weight $\tau'\in{\mathcal
  B}(\lambda')$ either (i) there is a unique weight
$\tau^+\in{\mathcal B}(\lambda^+)\cap\supp(\tau')$ and $\tau'$ is the
unique weight in ${\mathcal B}(\lambda')\cap\supp(\tau^+)$, or (ii)
there exists $\tau^-,\tau^+\in{\mathcal B}(\lambda^+)$ such that
$\tau'$ separates $\tau^-$ and $\tau^+$. Then we say that $\lambda'$
is in the \emph{lower closure} of $\lambda^+$. If every pair of
weights $\mu^-$ and $\mu^+$ in ${\mathcal B}(\lambda^+)$ separated by
some $\mu'\in{\mathcal B}(\lambda')$ satisfy the condition in Theorem
\ref{local}(ii) then we say that ${\mathcal B}(\lambda^+)$ has
\emph{enough local homomorphisms} with respect to ${\mathcal B}(\lambda')$.

We will need one new general result about translation functors for
towers of recollement not included in \cite{cdm3}. 

\begin{prop}\label{lc} Suppose that ${\mathcal B}(\lambda^+)$ has
enough local homomorphisms with respect to ${\mathcal B}(\lambda')$.
If $\lambda'$ is in the lower closure of $\lambda^+$ then
$$\ind^{\lambda^+}_n P_n(\lambda')\cong P_{n+1}(\lambda^+).$$
If further $\lambda'\in\Lambda_{n-2}$ then
$$\res^{\lambda^+}_n P_n(\lambda')\cong P_{n-1}(\lambda^+).$$
\end{prop}
\begin{proof}
The module $\ind_n^{\lambda^+} P_n(\lambda')$ is clearly
projective, as induction (and taking a direct summand) takes
projectives to projectives. 

Suppose that $\tau\in{\mathcal B}(\lambda^+)$ and that
$$\res_{n+1}^{\lambda'}L_{n+1}(\tau)\neq 0.$$ By our
assumptions and Theorems \ref{equiv} and \ref{local} this implies that
$\tau\in\supp(\mu')$ for some $\mu'\in{\mathcal B}(\lambda')$ and
$\tau=\mu^+$. From this we see that if
$$\Hom_{n+1}(\ind^{\lambda^+}_nP_n(\lambda'),L_{n+1}(\tau))
=\Hom_n(P_n(\lambda'),\res^{\lambda'}_{n+1}L_{n+1}(\tau))$$
is non-zero then $\tau=\mu^+$ for some $\mu'\in{\mathcal
  B}(\lambda')$. But 
$$\Hom_n(P_n(\lambda'),\res^{\lambda'}_{n+1}L_{n+1}(\mu^+))
=\Hom_n(P_n(\lambda'),\pr_n^{\lambda'}L(\mu'))=\delta_{\lambda'\mu'}$$
by Theorems \ref{equiv} and \ref{local}. Thus 
$\ind_n^{\lambda^+}P_n(\lambda')$ has simple head
$L_{n+1}(\lambda^+)$, and hence is equal to $P_{n+1}(\lambda^+)$ as required.

Now suppose that further $\lambda'\in\Lambda_{n-2}$. By \cite[Lemma
  4.3]{cdm3} we have 
$$G_{n-2}P_{n-2}(\lambda')\cong P_{n}(\lambda').$$
By the tower of recollement axioms we have
$$\ind_{n-2}^{\lambda^+}M\cong\res_n^{\lambda^+}G_{n-2}M$$
for any $A_{n-2}$-module $M$ and hence
$$\res_n^{\lambda^+}P_n(\lambda')
\cong\res_n^{\lambda^+}G_{n-2}P_{n-2}(\lambda')
\cong \ind^{\lambda^+}_{n-2}P_{n-2}(\lambda')\cong
P_{n-1}(\lambda^+)$$
using the first part of the Proposition.
\end{proof}

\begin{rem}
It was shown in \cite{cdm} that the Brauer algebras form a tower of
recollement. Similarly, in \cite[Sections 2-3]{cddm} it was shown that
the walled Brauer algebras form a tower of recollement by using
alternately the functors $\res^L$ (and $\ind^L$) and $\res^R$ (and
$\ind^R$). The existence of enough local homomorphisms was shown for
the Brauer algebra in \cite[Theorem 3.4]{dhw} and for the walled
Brauer algebra in \cite[Theorem 6.2]{cddm}. Thus we can apply the
results of this section to these algebras. 

When using the notation
$\ind_{r,s}^{\lambda}$ for the walled Brauer algebra, the choice of
$\ind_{r,s}^L$ or $\ind_{r,s}^R$ will be such that the weight
$\lambda$ makes sense for the resulting algebra (and similarly for
$\res_{r,s}^{\lambda}$).
\end{rem}

\begin{rem}
There are reflection geometries controlling the block structure of the
Brauer \cite{cdm2} and walled Brauer algebras \cite{cddm} which we
will review shortly. These define a system of facets, and in
\cite{cdm2} it was shown that two weights in the same facet for the
Brauer algebra have weakly Morita equivalent blocks in the sense of
Theorem \ref{walls}. This required certain generalised induction and
restriction functors for the non-alcove cases. Similar functors can be
defined for the walled Brauer algebras: it is a routine but lengthy
exercise to verify that the construction in \cite[Section 5]{cdm3} can
be extended to the walled Brauer case. Thus we also have weak Morita
equivalences between weights in the same facet in the walled Brauer
case.
\end{rem}

\section{Oriented cap diagrams}\label{capsec}

In this section we will describe the construction of oriented cap
diagrams associated to certain pairs of weights for the walled Brauer
algebra. These diagrams were introduced by Brundan and Stroppel in \cite{bs1}
to study Khovanov's diagram algebra. We will see later that they give
precisely the combinatoric required to describe decomposition numbers
for the walled Brauer algebra.

Let $\{\epsilon_i:i\in\ZZ, i\neq0\}$ be a set of formal symbols, and
set
$$X=\prod_{i\in\ZZo}\ZZ\epsilon_i.$$
For $x\in X$ we write 
$$x=(\ldots,x_{-3},x_{-2},x_{-1};x_1,x_2,x_3,\ldots)$$
where $x_i$ is the coefficient of $\epsilon_i$. We define $A^+\subset
X$ by
$$A^+=\{x\in X : \cdots >x_{-3}>x_{-2}>x_{-1}, x_1>x_2>x_3>\cdots\}$$
and for $\delta\in\ZZ$ we define
$$\rho=\rho_\delta=(\cdots,3,2,1;\delta,\delta-1,\delta-2,\cdots)\in A^+.$$

Given a bipartition $\lambda=(\lambda^L,\lambda^R)$ with
$\lambda^L=(\lambda_1^L,\ldots,\lambda_r^L)$ and 
$\lambda^R=(\lambda_1^R,\ldots,\lambda_s^R)$, we define
$\bar{\lambda}\in X$ by
$$\bar{\lambda}=(\ldots,0,0,-\lambda^L_r,-\lambda^L_{r-1},\ldots,-\lambda^L_1;
\lambda_1^R,\ldots,\lambda_s^R,0,0,\ldots).$$
 Given such a bipartition
$\lambda$ we define
$$x_{\lambda}=x_{\lambda,\rho}=\bar{\lambda}+\rho_{\delta}.$$ Note
 that $x_{\lambda}\in A^+$. In this way we can embed the sets
 $\Lambda_{r,s}$ labelling simple modules for $B_{r,s}(\delta)$ as subsets of
 $A^+$.

Consider the group $W$ of all permutations of finitely many elements from
the set $\ZZo$ (so $W=\langle (i,j): i,j\in\ZZo\rangle$ where $(ij)$
is the usual notation for transposition of a pair $i$ and $j$). This
group acts on $X$ by place permutations.

The main result (Corollary 10.3) in \cite{cddm} describes the blocks of
$B_{r,s}(\delta)$ in terms of orbits of certain finite reflection
groups inside $W$. However it is easy to see from the proof that the
following version also holds.

\begin{thm}
Two simple modules $L_{r,s}(\lambda)$ and $L_{r,s}(\mu)$
are in the same block if and only if $x_{\lambda}=wx_{\mu}$ for some
$w\in W$.
\end{thm}

We will abuse terminology and say that $x_{\lambda}$ and $x_{\mu}$ are
in the same block if they satisfy the conditions of this theorem.

To each element $x\in A^+$ we wish to associate a diagram with
vertices indexed by $\ZZ$, each labelled with one of the symbols
$\circ$, $\times$, $\wedge$, $\vee$. We do this in the following
manner. Given $x\in A^+$ define
$$I_{\vee}(x)=\{x_i:i<0\}\quad\text{and}\quad
I_{\wedge}(x)=\{x_i:i>0\}.$$
Now vertex $n$ in the diagram associated to $x$ is labelled by
\begin{equation}\label{assign}
\left\{\begin{array}{cl}
\circ &\text{if}\ n\notin I_{\vee}(x)\cup I_{\wedge}(x)\\
\times&\text{if}\ n\in I_{\vee}(x)\cap I_{\wedge}(x)\\
\vee&\text{if}\ n\in I_{\vee}(x)\backslash I_{\wedge}(x)\\
\wedge&\text{if}\ n\in I_{\wedge}(x)\backslash
I_{\vee}(x).\end{array}\right.
\end{equation}

\begin{example}\label{bi}
To illustrate the above construction, consider the bipartition
$\lambda=(\lambda^L,\lambda^R)$ where $\lambda^L=(2,2,1)$ and
$\lambda^R=(3,2)$, and take $\delta=2$. Then
$$\rho_{\delta}=(\ldots,4,3,2,1;2,1,0,-1,-2,\ldots)$$
and 
$$\bar{\lambda}=(\ldots,0,-1,-2,-2;3,2,0,0,0,\ldots)$$
and hence
$$x_{\lambda}=\bar{\lambda}+\rho_{\delta}=
(\ldots,6,5,4,2,0,-1;5,3,0,-1,-2,-3\ldots).$$
Part of the associated diagram is illustrated in Figure \ref{basic}.
\end{example}

\begin{figure}[ht]
\includegraphics{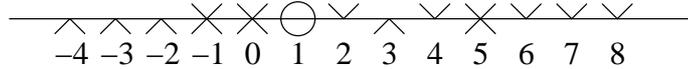}
\caption{The diagram associated to $((2,2,1),(3,2))$ with $\delta=2$.}
\label{basic}
\end{figure}

Note that any element in $A^+$ is uniquely determined by its diagram,
and every such diagram corresponds to an element in $A^+$. For this
reason we will use the notation $x$ (or $x_{\lambda}$) for both.

\begin{rem}
It is easy to see that two elements in $A^+$ are in the same $W$-orbit
if and only if they are obtained from each other by permuting pairwise
a finite number of $\wedge$s and $\vee$s.
\end{rem}

We define a partial order $\leq$ on $A^+$ by setting $x<y$ if $y$ is
obtained from $x$ by swapping a $\vee$ and a $\wedge$ so that the
$\wedge$ moves to the right, and extending by transitivity. Note that
if $\lambda,\mu\in \Lambda_{r,s}$ then $x_{\lambda}\leq x_{\mu}$ if
and only if $\lambda$ and $\mu$ are in the same block and $\lambda\leq
\mu$ (where this is the natural order on bipartitions from Section
\ref{basics}). Therefore we use the same symbol for both orders.

\begin{example} There is only one element in $A^+$ smaller than the
  element $x_{\lambda}$ in Example \ref{bi}. This corresponds to the
  diagram in Figure \ref{min}.
\end{example}

\begin{figure}[ht]
\includegraphics{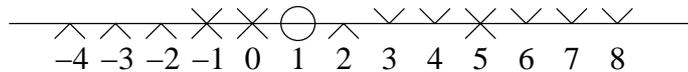}
\caption{The unique diagram smaller than the diagram in Figure \ref{basic}.}
\label{min}
\end{figure}

\begin{rem} 
For a bipartition $\lambda=(\lambda^L,\lambda^R)$, the diagram for the
element $x_{\lambda}\in A^+$ is labelled by $\wedge$ for all $n<<0$
and by $\vee$ for all $n>>0$. Thus there are only finitely many
$x<x_{\lambda}$. 
\end{rem}

To each bipartition $\lambda$ (or to each diagram labelled by $\wedge$
for all $n<<0$ and by $\vee$ for all $n>>0$) we associate a \emph{cap
  diagram} $c_{\lambda}$ in the following (recursive) manner.

In $x_{\lambda}$ find a pair of vertices labelled $\vee$ and $\wedge$
in order from left to right that are neighbours in the sense that
there are only $\circ$s, $\times$s, or vertices already joined by caps
at an earlier stage between them. Join this pair of vertices together with a
cap. Repeat this process until there are no more such $\vee$ $\wedge$
pairs. (This will occur after a finite number of steps.) Finally, draw
an infinite ray upwards at all remaining $\wedge$s and $\vee$s. Any
vertices which are not connected to a ray or a cap are called free
vertices.

\begin{example}
In Figures \ref{exa} and \ref{exb} we give two examples of elements
$x_{\lambda}$ and their associated cap diagrams.
\end{example}

\begin{figure}[ht]
\includegraphics{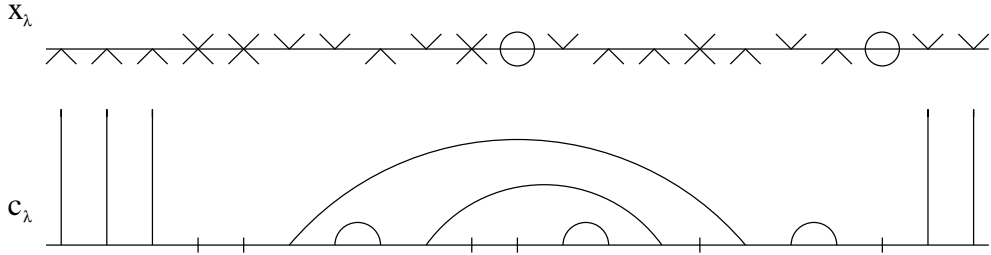}
\caption{An example of the cap diagram construction.}
\label{exa}
\end{figure}

\begin{figure}[ht]
\includegraphics{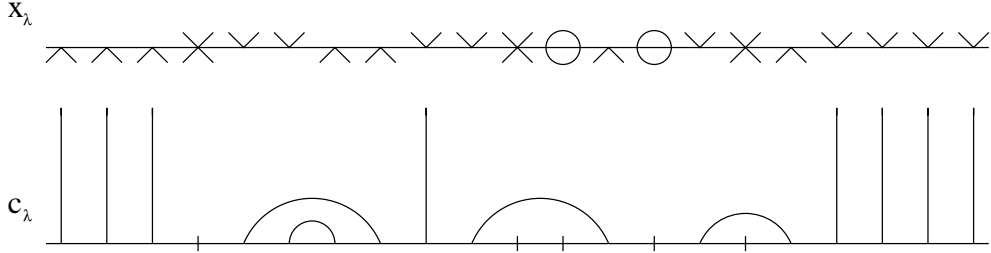}
\caption{Another example of the cap diagram construction.}
\label{exb}
\end{figure}

To a cap diagram $c$ and an element $x_\lambda\in A^+$ we can
associate a \emph{labelled cap diagram} $cx_{\lambda}$ by writing each label
on a vertex of $x_{\lambda}$ underneath the corresponding vertex of
$c$. We call such a diagram an \emph{oriented cap diagram} if the
following conditions all hold:
\begin{enumerate}
\item each free vertex in $c$ is labelled by a $\circ$ or $\times$ in
  $x_{\lambda}$;
\item the vertices at the end of each cap in $c$ are labelled by
  exactly one $\wedge$ and one $\vee$ in $x_{\lambda}$;
\item each vertex at the bottom of a ray in $c$ is labelled by a
  $\wedge$ or $\vee$ in $x_{\lambda}$;
\item it is impossible to find two rays in $c$ whose vertices are
  labelled $\vee$ and $\wedge$ in order from left to right in
  $x_{\lambda}$.
\end{enumerate}

As each cap in an oriented cap diagram is labelled by exactly one
$\wedge$ and one $\vee$, these symbols induce an orientation on the
cap (as though they were arrows). The \emph{degree}
$\deg(cx_{\lambda})$ of an oriented cap diagram $cx_{\lambda}$ is the
total number of clockwise caps that it contains.

\begin{rem}
Given a bipartition $\lambda$, the labelled cap diagram
$c_{\lambda}x_{\lambda}$ is clearly oriented, with all caps having a
counterclockwise orientation. Thus the degree of
$c_{\lambda}x_{\lambda}$ is $0$.
\end{rem}

For two bipartitions $\lambda$ and $\mu$ we define $d_{\lambda\mu}(q)$
to be $q^{\deg(c_{\lambda}x_{\mu})}$ if (i) $\lambda$ and $\mu$ are in
the same $W$-orbit, and (ii)
$c_{\lambda}x_{\mu}$ is an oriented cap diagram. We define
$d_{\lambda\mu}(q)$ to be $0$ otherwise. In other words,
$d_{\lambda\mu}(q)\neq 0$ if and only if $x_{\mu}$ is obtained from
$x_{\lambda}$ by swapping the order of the elements in some of the
pairs $\vee$, $\wedge$ which are joined up in $c_{\lambda}$, and in
that case $\deg(c_{\lambda}x_{\mu})$ is the number of pairs whose
elements have been swapped.

\begin{example}
Let $x_{\lambda}$ and $c_{\lambda}$ be as in Figure \ref{exa}. For
$x_{\mu}$ as illustrated in Figure \ref{degcap} we see that
$c_{\lambda}x_{\mu}$ is an oriented cap diagram with
$\deg(c_{\lambda}x_{\mu})=3$. Hence we have that
$$d_{\lambda\mu}(q)=q^3.$$ 
\end{example}

\begin{figure}[ht]
\includegraphics{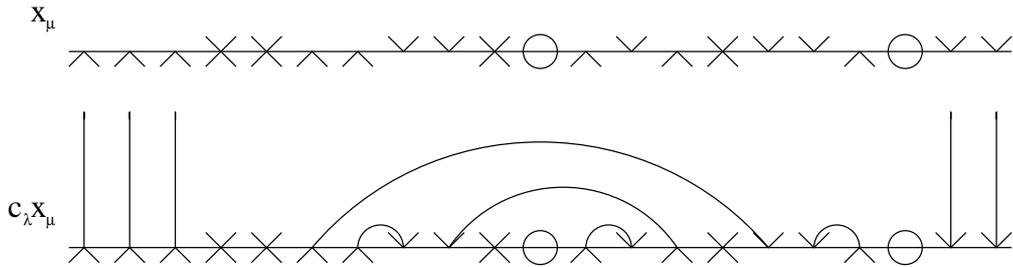}
\caption{An example of a nontrivial degree calculation.}
\label{degcap}
\end{figure}

\begin{rem}
Brundan and Stroppel have shown how to associate weights in a set
similar to $A^+$ to cap diagrams and oriented cap diagrams in order
to use this combinatoric to describe the representation theory to the
general linear supergroup GL$(m|n)$ \cite{bs4}. Note the difference
between these two sets, and the difference between the assignments of
labels in \cite[(1.6)]{bs4} and in (\ref{assign}).
\end{rem}

We are interested in determining the decomposition numbers for the
walled Brauer algebras. As noted in Section \ref{basics} this is equivalent
to determining the
$$D_{\lambda\mu}=(P_{r,s}(\lambda):\Delta_{r,s}(\mu)).$$
Our eventual aim is to show 

\begin{thm} \label{mainresult}
Given $\lambda$ and $\mu$ in $\Lambda_{r,s}$ we have
$$D_{\lambda\mu}=d_{\lambda\mu}(1).$$
\end{thm}

We will first introduce a corresponding formalism for the Brauer
algebra, so that the two cases can be considered simultaneously.

\section{Oriented curl diagrams}\label{twicapsec}

We will introduce analogues of oriented cap diagrams for use in the
ordinary Brauer algebra case. As the two cases will ultimately be very
similar, we use the same notation. Which case is being considered
later will be clear from context.

Let $\{\epsilon_i:i\in\NN\}$ be a set of formal symbols, and set 
$$X=\left(\prod_{i\in\NN}\ZZ\epsilon_i\right)\bigcup
\left(\prod_{i\in\NN}(\ZZ+\frac{1}{2})\epsilon\right).$$
For $x\in X$ we write 
$$x=(x_1,x_2,\ldots)$$
where $x_i$ is the coefficient of $\epsilon_i$. We define $A^+\subset
X$ by
$$A^+=\{x\in X: x_1>x_2>\cdots\}$$ 
and for $\delta\in\ZZ$ define 
$$\rho=\rho_{\delta}=(-\frac{\delta}{2},-\frac{\delta}{2}-1,
-\frac{\delta}{2}-2,-\frac{\delta}{2}-3,\ldots)\in A^+.$$
Given a partition $\lambda$ we define
$$x_{\lambda}=\lambda+\rho_{\lambda}\in A^+.$$

Consider the group 
$$W=\langle (i,j),(i,j)_-:i\neq j\in \NN\rangle$$ 
where $(ij)$ is the usual notation for transposition of a pair $i$ and
$j$, and $(i,j)_-$ is the element which transposes $i$ and $j$ and
also changes their signs. Then $W$ acts naturally on $X$, with $(ij)$
acting as place permutations, and 
$$(ij)_-(x_1,x_2,\ldots,x_i,\ldots,x_j,\ldots)
=(x_1,x_2,\ldots,-x_j,\ldots,-x_i,\ldots).$$

The main result in \cite{cdm2} describes the blocks of $B_n(\delta)$
in terms of certain finite reflection groups inside $W$. Just as in
the walled Brauer case, it is easy to see that the following version
holds. Here we denote the transpose of a partition $\lambda$ by
$\lambda^T$. 

\begin{thm}
Two simple modules $L_n(\lambda^T)$ and $L_n(\mu^T)$ are in the same
block if and only if $x_{\lambda}=wx_{\mu}$ for some $w\in W$.
\end{thm}

To each $x\in X$ we wish to associate a diagram. This will have
vertices indexed by $\NN\cup\{0\}$ if
$x\in\prod_{i\in\NN}\ZZ\epsilon_i$ or by $\NN-\frac{1}{2}$ if
$x\in\prod_{i\in\NN}(\ZZ+\frac{1}{2})\epsilon_i$. Each vertex will be
labelled with one of the symbols $\circ$, $\times$, $\vee$, $\wedge$,
or $\Diamond$. Given $x\in A^+$ define
$$I_{\wedge}(x)=\{x_i: x_i>0\}\quad \text{and}\quad I_{\vee}(x)=\{x_i:
x_i<0\}.$$ We also set $I_{\Diamond}(x)=\{x_i:x_i=0\}$, so
$I_{\Diamond}(x)$ can consist of at most one element. Now vertex $n$ in
the diagram associated to $x$ is labelled by
\begin{equation}\label{reassign}
\left\{\begin{array}{cl}
\circ &\text{if}\ n\notin I_{\vee}(x)\cup I_{\wedge}(x)\\
\times&\text{if}\ n\in I_{\vee}(x)\cap I_{\wedge}(x)\\
\vee&\text{if}\ n\in I_{\vee}(x)\backslash I_{\wedge}(x)\\
\wedge&\text{if}\ n\in I_{\wedge}(x)\backslash
I_{\vee}(x)\\
\Diamond &\text{if}\ n\in I_{\Diamond}(x).\end{array}\right.
\end{equation}

Note that every element in $A^+$ is uniquely determined by its
diagram, and every such diagram corresponds to an element in $A^+$
(provided that $0$ is labelled by $\circ$ or $\Diamond$). For this
reason we will use the notation $x$ (or $x_{\lambda}$) for both.

\begin{example}
Let $\lambda=(4,3,2)$ and $\delta=1$. Then we have
$$\rho_{\delta}=(-\frac{1}{2},-\frac{3}{2},-\frac{5}{2},\ldots,)$$
and
$$x_{\lambda}=\lambda+\rho_{\delta}=(\frac{7}{2},\frac{3}{2},
-\frac{1}{2},-\frac{7}{2},-\frac{9}{2},\ldots,).$$
The corresponding diagram is shown in Figure \ref{curlx}.
\end{example}

\begin{figure}[ht]
\includegraphics{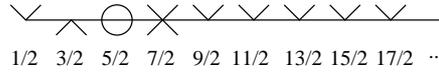}
\caption{The diagram associated to $\lambda=(4,3,2)$ when $\delta=1$.}
\label{curlx}
\end{figure}

\begin{rem}
It is easy to see that two elements in $A^+$ are in the same $W$-orbit
if and only if they are obtained from each other by repeatedly
swapping a $\vee$ and a $\wedge$ or replacing two $\vee$s by two
$\wedge$s, where $\Diamond$ can also play the role of either $\vee$ or
$\wedge$. If we fix a $\lambda$ where $x_{\lambda}$ contains
$\Diamond$ then we can arbitrarily choose to replace this $\Diamond$
by either $\vee$ or $\wedge$, and this defines a unique choice of
$\vee$ or $\wedge$ for every other element of the same
block. \emph{Thus in what follows we will always assume that a fixed
  choice of $\vee$ or $\wedge$ has been made for the symbol $\Diamond$
  for some weight in each block. Our combinatorial constructions will
  not be affected by this choice (provided we are consistent in a given
  block).}
\end{rem}

We define a partial order $\leq$ on $A^+$ by setting $x< y$ if $y$ is
obtained from $x$ by swapping a $\vee$ and a $\wedge$ so that the
$\wedge$ moves to the right, or if $y$ contains a pair of $\wedge$s
instead of a corresponding pair of $\vee$s in $x$, and extending by
transitivity. Note that for partitions $\lambda,\mu\in\Lambda_n$ we
have $x_{\lambda}\leq x_{\mu}$ if and only if $\lambda$ and $\mu$ are
in the same block and $\lambda\leq \mu$ (where this is the natural
order on partitions from Section \ref{basics}). Thus we use the same
symbol for both partial orders.

\begin{rem}
For a fixed partition $\lambda$ the diagram for $x_{\lambda}$ is
labelled by $\vee$ for all $n>>0$. Thus there are only finitely many
$x< x_{\lambda}$.
\end{rem}

To each $x_{\lambda}\in A^+$ we now associate a \emph{curl
  diagram} $c_{\lambda}$ in the following (recursive) fashion.

In $x_{\lambda}$ find a pair of vertices labelled $\vee$ and $\wedge$
in order from left to right that are neighbours in the sense that
there are only $\circ$s, $\times$s, or vertices already joined by caps
at an earlier stage between them. Join this pair of vertices together
with a cap. Repeat this process until there are no more such $\vee$
$\wedge$ pairs. (This will occur after a finite number of steps.)

Ignoring all $\circ$s, $\times$s and vertices on a cap, we are left
with a sequence of a finite number of $\wedge$s followed by an
infinite number of $\vee$s. Starting from the leftmost $\wedge$, join
each $\wedge$ to the next from the left which has not yet been used,
via a clockwise arc around all vertices to the left of the starting
vertex and without crossing any other arcs or caps. If there is a free
$\wedge$ remaining at the end of this procedure, draw an infinite ray
up from this vertex, and draw infinite rays from each of the remaining
$\vee$s. We will refer to the arcs connecting $\wedge$s as
\emph{curls}.

\begin{example}
An example of this construction is given in Figure \ref{exc}.
\end{example}

\begin{figure}[ht]
\includegraphics{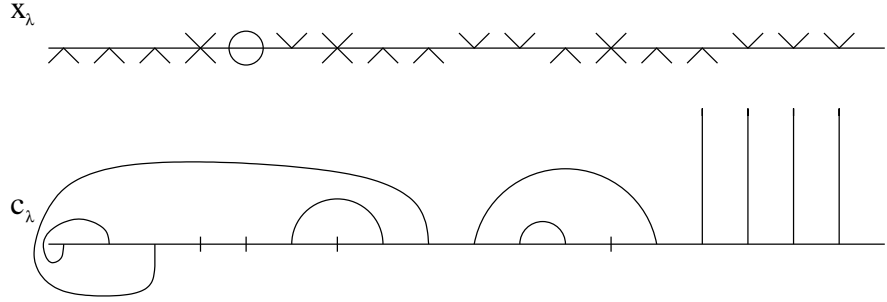}
\caption{An example of the curl diagram construction.}
\label{exc}
\end{figure}

To a curl diagram $c$ and an element $x_\lambda\in A^+$ we can
associate a \emph{labelled curl diagram} $cx_{\lambda}$ by writing each
label on a vertex of $x_{\lambda}$ underneath the corresponding vertex
of $c$. We call such a diagram an \emph{oriented curl diagram} if the
following conditions all hold:
\begin{enumerate}
\item each free vertex in $c$ is labelled by a $\circ$ or $\times$ in
  $x_{\lambda}$;
\item the vertices at the end of each cap in $c$ are labelled by
  exactly one $\wedge$ and one $\vee$ in $x_{\lambda}$;
\item the vertices at the end of each curl in $c$ are labelled by
  two $\wedge$s or two $\vee$s in $x_{\lambda}$;
\item each vertex at the bottom of a ray in $c$ is labelled by a
  $\wedge$ or $\vee$ in $x_{\lambda}$;
\item it is impossible to find two rays in $c$ whose vertices are
  labelled $\vee$ and $\wedge$, or $\wedge$ and $\wedge$, in order from
  left to right in $x_{\lambda}$.
\end{enumerate}

Each cap or curl in an oriented curl diagram has an orientation
induced by the terminal symbols (as though they were arrows). The
\emph{degree} $\deg(cx_{\lambda})$ of an oriented curl diagram
$cx_{\lambda}$ is the number of clockwise caps and curls that it contains.

\begin{rem}
Given a partition $\lambda$, all caps and curls in the labelled curl
diagram $c_{\lambda}x_{\lambda}$ are clearly oriented
anticlockwise. Thus the degree of $c_{\lambda}x_{\lambda}$ is $0$.
\end{rem}

For two partitions $\lambda$ and $\mu$ we define $d_{\lambda\mu}(q)$
to be $q^{\deg(c_{\lambda}x_{\mu})}$ if (i) $\lambda$ and $\mu$ are in
the same $W$-orbit, and (ii)
$c_{\lambda}x_{\mu}$ is an oriented curl diagram. We define
$d_{\lambda\mu}(q)$ to be $0$ otherwise. 

\begin{example}
Let $x_{\lambda}$ and $c_{\lambda}$ be as in Figure \ref{exc}. For
$x_{\mu}$ as illustrated in Figure \ref{degcurl} we see that
$c_{\lambda}x_{\mu}$ is an oriented curl diagram with
$\deg(c_{\lambda}x_{\mu})=2$. Hence we have that
$$d_{\lambda\mu}(q)=q^2.$$ 
\end{example}

\begin{figure}[ht]
\includegraphics{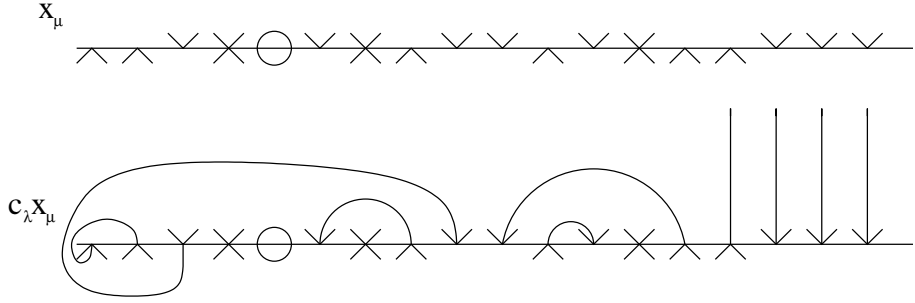}
\caption{An example of the calculation of the degree of a curl diagram.}
\label{degcurl}
\end{figure}

We are interested in determining the decomposition numbers for the
Brauer algebras (and hence recovering the result of Martin
\cite{marbrauer}). As noted in Section \ref{basics}
this is equivalent to determining the 
$$D_{\lambda\mu}=(P_{n}(\lambda):\Delta_{n}(\mu)).$$
Our eventual aim is to show 

\begin{thm} \label{mainresult2}
Given $\lambda$ and $\mu$ in $\Lambda_{n}$ we have
$$D_{\lambda\mu}=d_{\lambda\mu}(1).$$
\end{thm}

\section{Decomposition numbers from oriented cap and curl diagrams}
\label{repthry}

The aim of this section is to prove Theorems \ref{mainresult} and
\ref{mainresult2}. To do this we will apply the translation functor
formalism from Section \ref{tor}. We will consider the two cases
simultaneously as they are very similar.

Fix $\lambda\in\Lambda_{r,s}$ or $\Lambda_n$. We will proceed by
induction on the partial order $\leq$ introduced in Section
\ref{capsec} or \ref{twicapsec}. If $x_\lambda$ is minimal in its
block with respect to the order $\leq$ then we have
$$D_{\lambda\mu}=\delta_{\lambda\mu}=d_{\lambda\mu}(1)$$
for all $\mu$ and so we are done.

Suppose that $x_{\lambda}$ is not minimal in its block. We proceed by
induction on $|\lambda|$. (Note that if $\lambda^L=\emptyset$ or
$\lambda^R=\emptyset$ then
$x_{\lambda}=\rho$ is minimal.) Then $x_{\lambda}c_{\lambda}$ contains
at least one cap or curl.

First consider the cap case: we may choose the cap so that it does not
contain any smaller caps (and hence all vertices inside the cap are
labelled by $\times$ or $\circ$ only). We call such a cap a
\emph{small} cap. There are three cases, which
are illustrated in Figure \ref{capcase}. Note that we will henceforth
abuse notation and write $\lambda$ instead of $x_{\lambda}$. 

\begin{figure}[ht]
\includegraphics{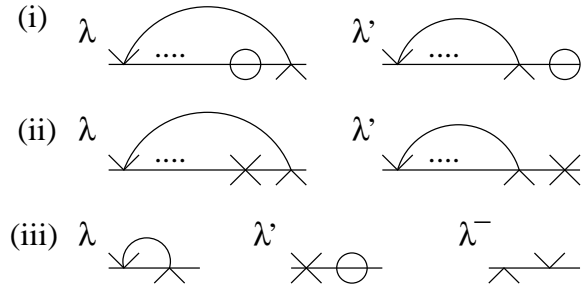}
\caption{The three possible small cap configurations}
\label{capcase}
\end{figure}

\noindent{\bf Case (i):} The vertex at the point marked with a
$\wedge$ is of the form $x_i$ for some $i\in\ZZo$, and (in the walled
Brauer case) by the definition of $\wedge$ we must have $i>0$. Now
consider $\lambda'=(\lambda^l,\lambda^R-\epsilon_i)$ or
$\lambda'=\lambda-\epsilon_i$. Note that $x_i-1$ is not an entry in
$x_{\lambda}$ and hence $\lambda'$ is a (bi)partition.  The diagram
associated to $\lambda'$ is illustrated on the right-hand side of
Figure \ref{capcase}(i).

We claim that $\lambda$ and $\lambda'$ are translation equivalent;
that is for every $\mu\in{\mathcal B}(\lambda)$ there exists a unique
$\mu'\in{\mathcal B}(\lambda')\cap\supp(\mu)$ and for every
$\mu'\in{\mathcal B}(\lambda')$ there is a unique $\mu\in{\mathcal
  B}(\lambda)\cap\supp(\mu')$. Indeed, it is easy to see that the
only places where $x_{\mu}$ and $x_{\mu'}$ can differ are at the
vertices labelled $x_i$ and $x_{i}-1$, and the possible cases are
illustrated in Figure \ref{trani}.

\begin{figure}[ht]
\includegraphics{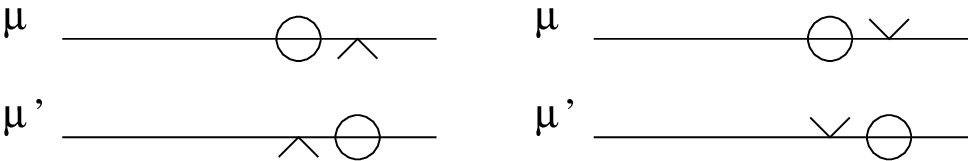}
\caption{The possible diagrams for $x_{\mu}$ and $x_{\mu'}$ in case (i)}
\label{trani}
\end{figure}

Recall our labelling convention involving $(a)$ from Section
\ref{basics}. By Theorem \ref{equiv} and the inductive hypothesis we
have that
\begin{equation}\label{this}
\begin{array}{lll}D_{\lambda\mu}&=&[\Delta_{(a)}(\mu):L_{(a)}(\lambda)]\\
&=&
  [\Delta_{(a-1)}(\mu'):L_{(a-1)}(\lambda')]=
D_{\lambda'\mu'}=d_{\lambda'\mu'}(1).
\end{array}\end{equation}
But if we ignore the $\times$s and $\circ$s (which play no role other
than as place markers in the
definition of $d_{\lambda\mu}$) then the cap or curl diagrams $c_{\lambda}$
and $c_{\lambda'}$ are identical, and hence 
\begin{equation}\label{that}
d_{\lambda'\mu'}(1)=d_{\lambda\mu}(1).
\end{equation}
Combining (\ref{this}) and (\ref{that}) we see that
$D_{\lambda\mu}=d_{\lambda\mu}(1)$ as required.

\noindent{\bf Case (ii):} This is very similar to case (i). The vertex
at the point marked with a $\times$ is in the walled Brauer case of
the form $x_{-j}$ for some $j\in\ZZo$, and by the definition of
$\times$ we can take $j>0$. In the Brauer case this vertex is of the
form $x_i>0$ and $x_i$ and $-x_{i}$ both appear in $x_{\lambda}$, and
we choose $j$ so that $x~_j=-x_i$.

Now consider $\lambda'=(\lambda^L-\epsilon_j,\lambda^R)$ or
$\lambda'=\lambda-\epsilon_j$. (As before it is easy to verify that
$\lambda'$ is a (bi)partition.) The diagram associated to $\lambda'$ is
illustrated on the right-hand side of Figure \ref{capcase}(ii). As in
case (i) the weights $\lambda$ and $\lambda'$ are translation
equivalent, where the various possibilities for $x_{\mu}$ and
$x_{\mu'}$ as before are shown in Figure \ref{tranii}.

\begin{figure}[ht]
\includegraphics{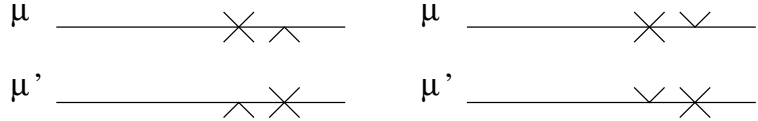}
\caption{The possible diagrams for $x_{\mu}$ and $x_{\mu'}$ in case (ii)}
\label{tranii}
\end{figure}

The rest of the argument proceeds exactly as in case (i).

\noindent{\bf Case (iii):} The vertex at the point marked with a
$\wedge$ is of the form $x_{i}$ for some $i\in\ZZo$, and (in the
walled Brauer case) by the definition of $\wedge$ we must have
$i>0$. Now consider $\lambda'=(\lambda^L,\lambda^R-\epsilon_i)$ or
$\lambda'=\lambda-\epsilon_i$ (which as before is a (bi)partition),
and set $\lambda^+=\lambda$. Note that there is another element
$\lambda^-\in{\mathcal B}(\lambda^+)\cap\supp(\lambda')$; the three
diagrams associated to $\lambda^+$, $\lambda'$ and $\lambda^-$ are
illustrated in Figure \ref{capcase}(iii).

Moreover, for each $\mu'\in{\mathcal B}(\lambda')$ there are exactly
two elements $\mu^+$ and $\mu^-$ in ${\mathcal
  B}(\lambda)\cap\supp(\mu')$ (which correspond to the same three
configurations as for $\lambda^+$, $\lambda^-$, and $\lambda'$ at the
two points $x_i$ and $x_i-1$). Also, $\mu'$ is the unique element in
${\mathcal B}(\lambda')\cap\supp(\mu^{\pm})$. Thus $\lambda'$ is in
the lower closure of $\lambda^+$.

For $\mu\in{\mathcal B}(\lambda)$ we have
$$\begin{array}{lcl}
D_{\lambda\mu}&=&[\Delta_{(a)}(\mu):L_{(a)}(\lambda)]\\
&=&\dim\Hom(P_{(a)}(\lambda^+),\Delta_{(a)}(\mu))\\
&=&\dim\Hom(\ind_{(a-1)}^{\lambda}P_{(a-1)}(\lambda'),\Delta_{(a)}(\mu))\\
&=&\dim\Hom(P_{(a-1)}(\lambda'),\res_{(a)}^{\lambda'}\Delta_{(a)}(\mu))
\end{array}$$
where the third equality follows from Proposition \ref{lc}. Now
$\res_{(a)}^{\lambda'}\Delta_{(a)}(\mu)\neq 0$ implies that
$\mu=\mu^{\pm}$ with $\mu'\in{\mathcal B}(\lambda')\cap\supp(\mu^{\pm})$
and so
$D_{\lambda\mu}=0$ unless $\mu=\mu^{\pm}\in{\mathcal
  B}(\lambda)\cap\supp(\mu')$. Note that for any $\mu$ not of this
form in $\mathcal B(\lambda)$ the two vertices labelled $x_i$ and
$x_i-1$ must be either both $\wedge$s or both $\vee$s, which implies
that $d_{\lambda\mu}=0$.

If $\mu=\mu^{\pm}$ as above then
$$D_{\lambda\mu^{\pm}}=
\dim\Hom(P_{(a-1)}(\lambda'),\Delta_{(a-1)}(\mu'))
=D_{\lambda'\mu'}=d_{\lambda'\mu'}(1)$$
by the induction hypothesis. Finally, note that $c_{\lambda'}x_{\mu'}$
is an oriented cap diagram if and only if $c_{\lambda}x_{\mu^{\pm}}$
is an oriented cap diagram, and so
$$D_{\lambda\mu^{\pm}}=d_{\lambda\mu^{\pm}}(1)$$
as required.

This completes the proof for the walled Brauer algebra. However, for
the Brauer algebra the diagram $c_{\lambda}$ may contain only
curls. We pick the one involving the left-most $\wedge$, and 
there are five cases, which are illustrated in Figure
\ref{curlcase}.

\begin{figure}[ht]
\includegraphics{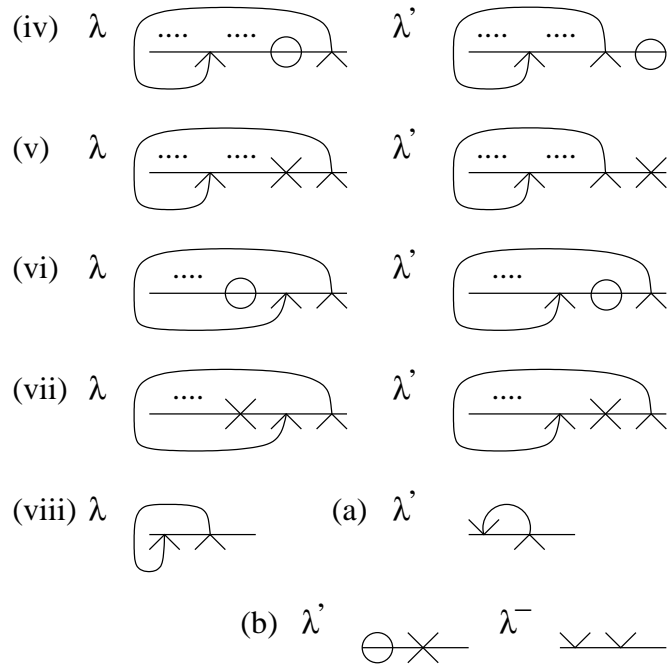}
\caption{The five possible small curl configurations}
\label{curlcase}
\end{figure}

\noindent{\bf Cases (iv-vii):} These are very similar to cases (i) and
(ii) above. In each case $\lambda'$ is obtained from $\lambda$ by
swapping the leftmost or right-most end of the curl with the symbol
immediately to its left (either $\circ$ or $\times$). Arguing exactly
as in cases (i) and (ii) we see that $\lambda$ and $\lambda'$ are translation
equivalent, and satisfy
$$d_{\lambda\mu}(q)=d_{\lambda'\mu'}(q).$$
Thus the result follows by induction.

\noindent{\bf Case (viii):} We are left with the case where the curl is
labelled with (a) $\frac{1}{2}$ and $\frac{3}{2}$, or (b) $0$ and $1$.

First consider configuration (a), with $\frac{1}{2}$ in the $i$th
entry of $x_{\lambda}$. As $-\frac{1}{2}$ is not in $x_{\lambda}$ we
have that $\lambda'=\lambda-\epsilon_i$ is a partition. The
corresponding diagrams are illustrated in Figure
\ref{curlcase}(viii)(a). These two elements are translation
equivalent, and the result follows as in case (i).
 
Finally consider configuration (b), and suppose that $0$ is in the
$i$th entry of $x_{\lambda}$. As $-1$ is not in $x_{\lambda}$, we have
that $\lambda'=\lambda-\epsilon_i$ is a partition. Setting
$\lambda^+=\lambda$ we see by arguing as in case (iii) that $\lambda'$
is in the lower closure of $\lambda^+$ (with $\lambda^-$ as
illustrated in Figure \ref{curlcase}(viii)(b)). The result for this
case follows just as in case (iii).

\begin{rem} We have shown that 
$$D_{\lambda\mu}=d_{\lambda\mu}(1)$$ 
for both the Brauer and walled Brauer algebras. In the Brauer case
Martin \cite{marbrauer} has introduced a similar diagram calculus, but
omitting the labels marked with $\times$ or $\circ$ and using caps
instead of curls. This allowed him to define versions of the
$d_{\lambda\mu}(q)$ and determine the decomposition numbers.

However, the $d_{\lambda\mu}(q)$ encode more than just their values at
$q=1$, and we would like to have a representation-theoretic
interpretation of these as polynomials in $q$. Instead we shall define
some closely related polynomials $p_{\lambda\mu}(q)$ and show how
these can be related to projective resolutions for our algebras. The
definition of this second family of polynomials crucially depends on
the distinction between caps and curls in our construction of curl diagrams.
\end{rem}

Before defining our second family of polynomials, we consider the
relation of the $d_{\lambda\mu}(q)$ to certain Kazhdan-Lusztig polynomials.

\section{A recursive formula for decomposition numbers}\label{recur}

We will show how the polynomials $d_{\lambda\mu}(q)$ can be calculated
recursively. The Brauer and walled Brauer cases will be considered
simultaneously. We will then relate this to the conjectured recursive
formula for the Brauer algebra given in \cite{cdm3} (and proved in
\cite{marbrauer}).

\begin{prop}\label{recp}
(i) Let $\lambda'\in\supp(\lambda)$ be as in one of the cases in
  Figure \ref{capcase} or \ref{curlcase}, with $\lambda$ and
  $\lambda'$ translation equivalent. Then
$$d_{\lambda'\mu'}(q)=d_{\lambda\mu}(q).$$ 
(ii) Suppose that $\lambda$ contains a small cap as in Figure
  \ref{capcase}(iii), or a small curl as in Figure
  \ref{curlcase}(viii) with $0$ in $x_{\lambda}$. Denote $\lambda$ by
  $\lambda^+$ and let $\lambda'$ and $\lambda^-$ be as indicated in
  the corresponding Figure. Then
$$d_{\lambda^+\mu^+}(q)=d_{\lambda'\mu'}(q)$$
and 
$$d_{\lambda^+\mu^-}(q)=qd_{\lambda'\mu'}(q).$$
Also we have
\begin{equation}\label{s1}
d_{\lambda^+\mu^+}(q)=q^{-1}d_{\lambda^-\mu^+}(q)+d_{\lambda^-\mu^-}(q)
\end{equation}
and
\begin{equation}\label{s2}
d_{\lambda^+\mu^-}(q)=qd_{\lambda^-\mu^-}(q)+d_{\lambda^-\mu^+}(q).
\end{equation}
\end{prop}

\begin{proof}
Everything is obvious by construction except for (\ref{s1}) and
(\ref{s2}). There are seven cases, which are illustrated in the cap
case in Figure \ref{four} and in the curl case in Figure \ref{free}.

\begin{figure}[ht]
\includegraphics{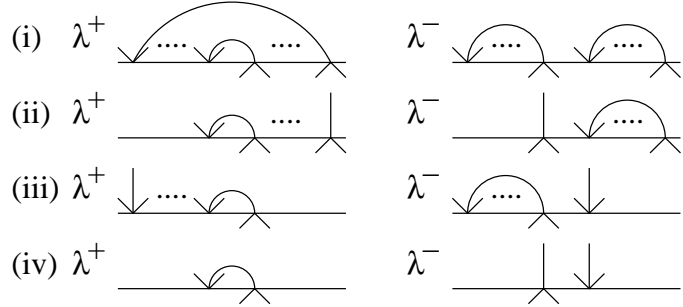}
\caption{Four small cap configurations}
\label{four}
\end{figure}

\begin{figure}[ht]
\includegraphics{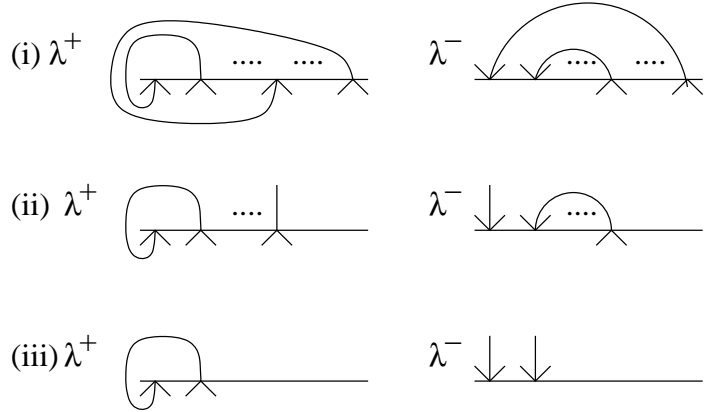}
\caption{Three small curl configurations}
\label{free}
\end{figure}

All of the cases are very similar, so we will consider just the case
in Figure \ref{four}(i). The weights $\lambda^+$ and $\lambda^-$ are
illustrated in Figure \ref{taster} together with the two possible
configurations (a) and (b) for $\mu^+$ and $\mu^-$ in the same block as
$\lambda^+$ and $\lambda^-$ at the four marked vertices. (The elements
$\mu$ and $\mu'$ must agree at all of the vertices not indicated in
the diagram.)

\begin{figure}[ht]
\includegraphics{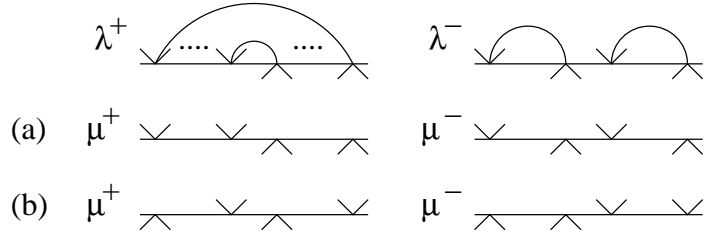}
\caption{The first cap case}
\label{taster}
\end{figure}

If $\mu^+$ is as in configuration (a) then we have
\begin{eqnarray*}
d_{\lambda^+\mu^-}(q)&=&qd_{\lambda^+\mu^+}(q)\\
d_{\lambda^-\mu^-}(q)&=&d_{\lambda^+\mu^+}(q)\\
d_{\lambda^-\mu^+}(q)&=&0
\end{eqnarray*}
which implies (\ref{s1}) and (\ref{s2}) as required.

If $\mu^+$ is as in configuration (b) then we have
\begin{eqnarray*}
d_{\lambda^+\mu^-}(q)&=&qd_{\lambda^+\mu^+}(q)\\
d_{\lambda^-\mu^-}(q)&=&0\\
d_{\lambda^-\mu^+}(q)&=&qd_{\lambda^+\mu^+}(q)
\end{eqnarray*}
which implies (\ref{s1}) and (\ref{s2}) as required. Similar arguments
hold in the remaining cases.
\end{proof}

Suppose that $\lambda$ is a regular weight, i.e. there are no vertices
labelled $\times$ in its diagram. This corresponds to
$\lambda$ lying in an alcove in the language of \cite{cddm} and
\cite{cdm2}. In \cite{cdm3} we reviewed the recursive formula for
parabolic Kazhdan-Lusztig polynomials of type $(D,A)$ following
\cite{soergel1} and conjectured that this gave the decomposition
numbers for the Brauer algebra. This algorithm is in two stages,
corresponding to translating the original polynomial and then
subtracting lower order terms.  This conjecture was proved by Martin
in \cite{marbrauer}. Exactly the same construction and conjecture can
be made for the walled Brauer case, involving parabolic
Kazhdan-Lusztig polynomials of type $(A,A\times A)$.

\begin{cor}
The decomposition numbers for the Brauer and walled Brauer algebras in
the case of regular blocks can be calculated (as parabolic
Kazhdan-Lusztig polynomials) as in \cite{cdm3}.
\end{cor}
\begin{proof}
It follows from (\ref{s1}) and (\ref{s2}) that the recursive formula
corresponding to translating a parabolic Kazhdan-Lusztig polynomial
holds for the $d_{\lambda\mu}$. By definition, the $d_{\lambda\mu}$
are monomials in $q$ with strictly positive degree if $\lambda\neq
\mu$ and $d_{\lambda\mu}(q)\neq 0$. This implies that there is no
subtraction of lower order terms in the calculation of parabolic
Kazhdan-Lusztig polynomials, and hence the $d_{\lambda\mu}(q)$ are
indeed parabolic Kazhdan-Lusztig polynomials.
\end{proof}

\begin{rem}
There are a number of related constructions of (parabolic)
Kazhdan-Lusztig polynomials (see \cite[Section 3]{soergel1} for the
relationship between them). In \cite{lsgrass} and \cite{boe} closed
forms are given for certain Kazhdan-Lusztig polynomials arising from
types $(D,A)$ and $(A,A\times A)$ (among others); in Section
\ref{vcap} we will recover these from our diagrams by defining new
polynomials $p_{\lambda\mu}(q)$. The relation between the
$p_{\lambda\mu}(q)$ and the $d_{\lambda\mu}(q)$ will be given in
Corollary \ref{pord}.

\end{rem}

\section{Valued cap and curl diagrams}\label{vcap}

In this section we will return to the combinatorics of cap and curl
diagrams, and define a new family of polynomials associated to pairs
of (bi)partitions $\lambda$ and $\mu$. These are given by a diagrammatic
version of the combinatorial formulas for Kazhdan-Lusztig
polynomials given in \cite{lsgrass} and \cite{boe}; a discussion of
the relation between the two approaches can be found in Appendix \ref{app}.

Fix $\lambda\in\Lambda_{r,s}$ or $\Lambda_n$ and
$\mu\in\calB=\calB(\lambda)$. We set $I(\calB)$ to be the infinite set
of non-zero integers indexing the vertices of $x_{\lambda}$ labelled
by $\vee$ or $\wedge$, but \emph{excluding} the leftmost one. Set
$I(\lambda,\mu)$ to be the finite subset of $I(\calB)$ indexing
vertices that are labelled differently in $x_{\lambda}$ and in
$x_{\mu}$. For $i\in I(\calB)$ define
$$\begin{array}{lcl}
l_i(\lambda,\mu)&=&
\#\{j\in I(\lambda,\mu): j\geq i\ \text{and vertex $j$ of
  $x_{\lambda}$ is labelled by $\wedge$}\}\\
&& -\#\{j\in I(\lambda,\mu): j\geq i\ \text{and vertex $j$ of
  $x_{\mu}$ is labelled by $\wedge$}\}.\end{array}$$
Note that $\lambda\geq\mu$ if and only if $l_i(\lambda,\mu)\geq 0$ for
all $i\in I(\calB)$. We set
$$l(\lambda,\mu)=\sum_{i\in I(\calB)}l_i(\lambda,\mu).$$

Any cap or curl diagram cuts the upper half plane into various open
connected regions, which we will call \emph{chambers}. Recall that we
say that a cap or curl in $c$ is small if it does not contain
any cap or curl inside it. Given a pair of chambers separated by a cap
or curl, we say that they are \emph{adjacent} and refer to the one
lying below as the \emph{inside} chamber, and the other as the
\emph{outside} chamber. The vertices labelled with $\vee$ or $\wedge$
will be called the \emph{non-trivial} vertices.

In the curl diagram case we may have a chamber $A$ (possibly
unbounded) inside which there are a series of maximal chambers
(i.e. chambers adjacent to $A$) $A_1,\ldots,A_t$ from left to right
not separated by the end of a curl. If $A_1$ is formed either by a
curl or by a cap involving the leftmost non-trivial vertex then we say
that $A_1,\ldots, A_t$ forms a \emph{chain}.

A \emph{valued cap diagram} $c$ is a cap diagram whose chambers have
been assigned values from the integers such that 
\begin{enumerate}
\item all external (unbounded)
chambers have value $0$;
\item given two adjacent chambers, the value of the inside chamber is
  at least as large as the value of the outer chamber.
\end{enumerate}
A \emph{valued curl diagram} $c$ is a curl diagram whose chambers have
been assigned values from the integers such that (1) and (2) above
hold and also
\begin{enumerate}
\item[(3)] the value of the chamber defined by a cap or curl connected
  to or containing inside itself the leftmost non-trivial vertex must be even;
\item[(4)] if $A_1,\ldots A_t$ is a chain and the value of $A_i$ is
  less than or equal to that of $A_j$ for all $1\leq j<i$ then the
  value of $A_i$ must be even. 
\end{enumerate}

Given a valued cap/curl diagram $c$, we write $|c|$ for the sum of
the values of $c$.

We are now able to define a new polynomial $p_{\lambda\mu}(q)$
associated to our pair $\lambda$ and $\mu$ in $\calB$. If
$x_{\lambda}\not\geq x_{\mu}$ then set
$p_{\lambda\mu}(q)=0$. Otherwise, let $D(\lambda,\mu)$ be the set of
all valued cap/curl diagrams obtained by assigning values to the chambers
of $c_{\mu}$ in such a way that the value of every small cap or curl is at
most $l_i(\lambda,\mu)$, where $i$ indexes the right-most vertex of the
cap or curl. Now set
$$p_{\lambda\mu}(q)=q^{l(\lambda,\mu)}\sum_{c\in
  D(\lambda,\mu)}q^{-2|c|}$$
and write $p_{\lambda\mu}^{(m)}$ for the coefficient of $q^m$ in
$p_{\lambda\mu}(q)$. That $p_{\lambda\mu}(q)$ is indeed a
polynomial will follow from Proposition \ref{klrel}, and hence
$$p_{\lambda\mu}(q)=\sum_{m\geq 0}p_{\lambda\mu}^{(m)}q^m.$$

\begin{example}
In Figure \ref{pcalc} we have illustrated a pair of diagrams
$x_{\lambda}$ and $x_{\mu}$ together with the curl diagram $c_{\mu}$
and the value of $l_i(\lambda,\mu)$ for each vertex $i$ in our
diagram. Thus in this case
$$l(\lambda,\mu)=2+3+2+2+1+1=11.$$
The various allowable values for the chambers in the curl
diagram are indicated in the Figure, where only the chambers marked
$a$ and $b$ can be non-zero. We must have $a\in\{0,2\}$ and $b\in\{0,1,2\}$.

Now the valued cap diagram is in
$D(\lambda,\mu)$  if and only if 
$$(a,b)\in\{(0,0),(0,1),(0,2),(2,0),(2,2)\}.$$
For example, note that we cannot have $(a,b)=(2,1)$ as this
configuration would not satisfy condition (4). Thus we see that
$$p_{\lambda\mu}(q)=q^{11}(1+q^{-2}+2q^{-4}+q^{-8})=q^{11}+q^9+2q^7+q^3.$$
\end{example}

\begin{figure}[ht]
\includegraphics{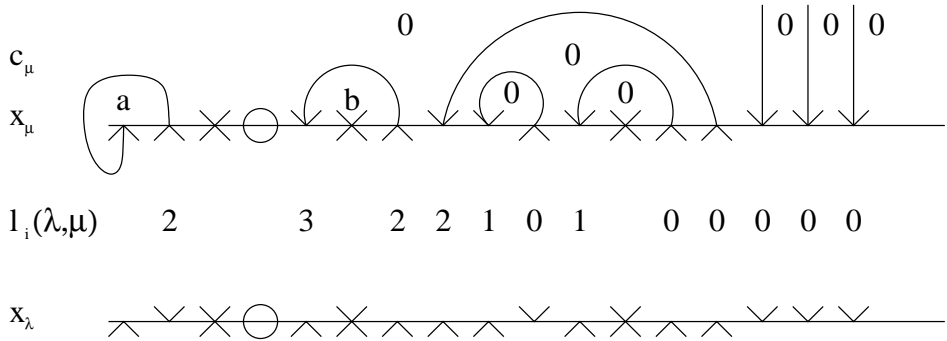}
\caption{An example of the calculation of $p_{\lambda\mu}(q)$.}
\label{pcalc}
\end{figure}

Pick a small cap or curl in $\lambda$. The possible configurations of
caps in $\lambda$ are given in Figure \ref{capcase}(i-iii) and of
curls in Figure \ref{curlcase}(iv-viii). Associated weights $\lambda'$ are
shown in each case, with two subcases appearing in Figure
\ref{curlcase}(viii), together with weights $\lambda^-$ in Figure
\ref{capcase}(iii) and Figure \ref{curlcase}(viii)(b). We will show how
the values of $p_{\lambda\mu}(q)$ can be calculated from the
polynomials $p_{\lambda'\mu'}$ and $p_{\lambda^-\tau}$ for suitable
choices of $\tau$, which will give a recursive formula for the
$p_{\lambda\mu}$.

Consider the configurations shown in Figure \ref{capcase}(iii) and
Figure \ref{curlcase}(viii)(b). In both of these cases we will denote
$\lambda$ by $\lambda^+$, and then the weights $\lambda^+$ and
$\lambda^-$ are separated by $\lambda'$ and $\lambda'$ is in the lower
closure of $\lambda^+$. We will say that an element is \emph{of the form
$\mu^+$} if it is in the same block as $\lambda^+$ and has the same
configuration of $\wedge$s and $\vee$s as $\lambda^+$ at the vertices on
the small cap or curl under consideration.

\begin{prop}\label{klrel}
(i) Let $\lambda$ and $\lambda'$ be one of the configurations in Figure
\ref{capcase}(i-ii) or Figure \ref{curlcase}(iv-vii), or 
as
in Figure \ref{curlcase}(viii)(a) where the vertices on the small curl
are labelled $\frac{1}{2}$ and $\frac{3}{2}$. Then
$$p_{\lambda\mu}(q)=p_{\lambda'\mu'}(q)$$ for all $\mu\in{\mathcal
  B}(\lambda)$. \\ 
(ii) Let $\lambda$ and $\lambda'$ be configured as
in Figure \ref{capcase}(iii) or as in Figure \ref{curlcase}(viii)(b)
where the vertices on the small curl are labelled $0$ and
$1$. Then setting $\lambda^+=\lambda$ we have
\begin{equation}\label{cross1}
p_{\lambda^+\mu^+}(q)=p_{\lambda'\mu'}(q)+qp_{\lambda^-\mu^+}(q)
\end{equation}
and
\begin{equation}\label{cross2}
p_{\lambda^+\mu}(q)=qp_{\lambda^-\mu}(q)
\end{equation}
for all $\mu$ not of the form $\mu^+$.
\end{prop}

\begin{proof} (Compare with \cite[(3.14) Proposition]{boe}.)
In the cases in Figure \ref{capcase}(i-ii) and Figure
\ref{curlcase}(iv-vii) the weights $\lambda$ and $\lambda'$ are
translation equivalent. By construction we have in all of these cases
that
$$p_{\lambda\mu}(q)=p_{\lambda'\mu'}(q)$$ for all $\mu\in{\mathcal
    B}(\lambda)$. 
The case in Figure \ref{curlcase}(viii)(a) occurs
when the vertices on the small curl are labelled $\frac{1}{2}$ and
$\frac{3}{2}$, and again the weights $\lambda$ and $\lambda'$ are
translation equivalent. The translation  equivalence is given by
changing the $\pm\frac{1}{2}$ entry in $x_{\mu}$ to $\mp\frac{1}{2}$
in $x_{\mu'}$. Therefore $l_i(\lambda,\mu)=l_i(\lambda',\mu')$ for all
$i\in I({\mathcal B})$ and all other caps and curls are
preserved. Thus in this case we also have that 
$$p_{\lambda\mu}(q)=p_{\lambda'\mu'}(q)$$ for all $\mu\in{\mathcal
    B}(\lambda)$. 

The two remaining cases are those shown in Figure \ref{capcase}(iii)
and Figure \ref{curlcase}(viii)(b). In both of these cases the weights
$\lambda^+$ and $\lambda^-$ are separated by $\lambda'$ and $\lambda'$
is in the lower closure of $\lambda^+$. We first consider
(\ref{cross1}). We claim there is a one-to-one correspondence between
$D(\lambda^+,\mu^+)$ and $D(\lambda',\mu')\sqcup
D(\lambda^-,\mu^+)$. Let $i$ be the rightmost vertex of the small cap
or curl under consideration in $x_{\lambda}$. It is easy to see that
$$l_i(\lambda^+,\mu^+)=l_i(\lambda^-,\mu^+)+1$$
and that if $i-1$ is the left-most non-trivial vertex then
$l_i(\lambda^+,\mu^+)$ is even. 

The valued cap/curl diagrams in $D(\lambda^+,\mu^+)$ split into two
subsets, those where the value of the small cap/curl under
consideration is less than $l_i(\lambda^+,\mu^+)$ and those where the
value is equal to $l_i(\lambda^+,\mu^+)$.  The first set are exactly
the valued cup/curl diagrams in $D(\lambda^-,\mu^+)$.  

We will show that the second set is obtained from the set of valued
cap/curl diagrams $D(\lambda',\mu')$ by adding to each element a
cap/curl joining vertices $i-1$ and $i$ with value
$l_i(\lambda^+,\mu^+)$. For $c\in D(\lambda',\mu')$ denote by $c^+$
the corresponding valued cap/curl diagram with this extra cap/curl. We
need to show that $c^+$ is indeed in $D(\lambda^+,\mu^+)$ to give the
desired bijection.

We check that inserting this extra cap/curl with the given value
satisfies the condition (1--4) in the definition of a valued cap/curl
diagram. (1) is obvious.

\begin{figure}[ht]
\includegraphics{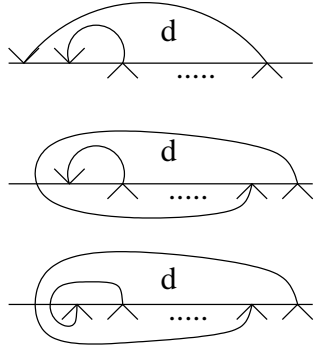}
\caption{The possible nested cases}
\label{nest1}
\end{figure}

For (2), suppose that our small cap/curl is nested inside a larger one
$d$. We may assume that they are adjacent. There are three possible
cases, illustrated in Figure \ref{nest1}. Suppose that there is a
small cap in the dotted region in Figure \ref{nest1}; if we pick the
leftmost such cap and $j$ denotes its right-hand vertex then it is easy
to see that
$$l_j(\lambda^+,\mu^+)\leq l_i(\lambda^+,\mu^+).$$
So the value of this small cap is at most $l_i(\lambda^+,\mu^+)$ and
hence the value of $d$ is at most $l_i(\lambda^+,\mu^+)$.

If the dotted region in Figure \ref{nest1} is empty then let $j$ be the
vertex at the right-hand end of the cap/curl defining $d$. If this is a
cap then we have 
$$l_j(\lambda^+,\mu^+)\leq l_i(\lambda^+,\mu^+)$$
and so the value of $d$ is at most $l_i(\lambda^+.\mu^+)$. If we have
a small cap or curl nested in a curl then
$$l_j(\lambda^+,\mu^+)\leq l_i(\lambda^+,\mu^+)+1.$$
But $d$ has to be even and $l_i(\lambda^+,\mu^+)$ is even, and so the
value of $d$ is at most $l_i(\lambda^+,\mu^+)$.

For (3), as noted above if $i-1$ is the leftmost non-trivial vertex
then $l_i(\lambda^+,\mu^+)$ is even.

Finally for (4), suppose we have a chain of chambers. If our small
cap/curl is the leftmost in the chain then denote the vertices of the
next chamber along in the chain as shown in Figure \ref{chain1}. By
the same argument as in (2) we see that $d$ has value at most
$l_i(\lambda^+,\mu^+)$, and as $k$ was the leftmost non-trivial vertex
we have that $d$ is even.

\begin{figure}[ht]
\includegraphics{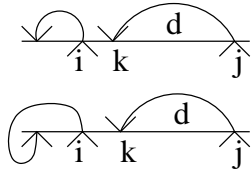}
\caption{The leftmost chain cases}
\label{chain1}
\end{figure}

If there is a chamber to each side of our small cap in the chain then
we are in the configuration shown in Figure \ref{chain2}. As before
the value of $e$ is at most $l_i(\lambda^+,\mu^+)$. If $e$ has value
at most that of $d$ and all other predecessors then removing the small
cap at $i$ we have a chain in $D(\lambda',\mu')$ and so $d$ is even as
required. 

\begin{figure}[ht]
\includegraphics{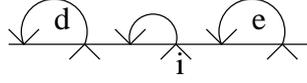}
\caption{The mid-chain cases}
\label{chain2}
\end{figure}

If our small cap is the rightmost in the chain then a similar argument
shows that the preceding chamber $d$ in the chain has value at most
$l_j(\lambda^+,\mu^+)\leq l_i(\lambda^+,\mu^+)$.  If
$l_i(\lambda^+,\mu^+)$ is no greater than all preceding values in the
chain then $l_i(\lambda^+,\mu^+)$ is at most the value of $d$, and
hence by the preceding inequality the value of $d$ equals
$l_i(\lambda^+,\mu^+)$. Removing our small cap gives a chain in
$D(\lambda',\mu')$ and hence $l_i(\lambda^+,\mu^+)$ must be even. Thus
conditions (1-4) are satisfied and hence $c^+\in D(\lambda^+,\mu^+)$
as required. 

It is also clear that
$$l(\lambda^+,\mu^+)=l(\lambda^-,\mu^+)+1$$
and
$$l(\lambda^+,\mu^+)=l(\lambda',\mu')+2l_i(\lambda^+,\mu^+).$$
Hence
\begin{eqnarray*}
p_{\lambda^+\mu^+}(q)&=&
q^{l(\lambda^+,\mu^+)}\sum_{c\in D(\lambda^+,\mu^+)}q^{-2|c|}\\
&=&q^{l(\lambda^+,\mu^+)}\sum_{c\in D(\lambda^-,\mu^+)}q^{-2|c|}
+q^{l(\lambda^+,\mu^+)}\sum_{c\in D(\lambda',\mu')}q^{-2|c^+|}\\
&=&q.q^{l(\lambda^-,\mu^+)}\sum_{c\in D(\lambda^-,\mu^+)}q^{-2|c|}
+q^{l(\lambda',\mu')+2l_i(\lambda^+,\mu^+)}
\sum_{c\in D(\lambda',\mu')}q^{-2|c|-2l_i(\lambda^+,\mu^+)}\\
&=& qp_{\lambda^-\mu^+}(q)+p_{\lambda'\mu'}(q).
\end{eqnarray*}

It remains to show that (\ref{cross2}) holds. If $\mu$ is not of the
form $\mu^+$ then it must have a different configuration of $\vee$s
and $\wedge$s on the pair of vertices defined by our small cap or
curl.  Thus the possible configurations are as indicated in Figure
\ref{nonplus}, where the top row (a-c) corresponds to the small cap
case in Figure \ref{capcase}(iii) and the bottom row (d-f) corresponds
to the small curl case in Figure \ref{curlcase}(viii)(b).

\begin{figure}[ht]
\includegraphics{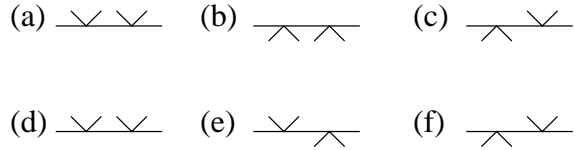}
\caption{The possible configurations of $\mu$ not of the form $\mu^+$}
\label{nonplus}
\end{figure}

In all six cases we have 
$$l(\lambda^+,\mu)=l(\lambda^-,\mu)+1.$$
Let $i$ be the rightmost of the vertices on the small cap/curl in
$\lambda$. Note that for all $j\neq i$ we have that
$$l_j(\lambda^+,\mu)=l_j(\lambda^-,\mu)\quad\text{and}\quad 
l_i(\lambda^+,\mu)=l_i(\lambda^-,\mu)+1.$$

Now for $\mu$ as in Figure \ref{nonplus}(a), (c), (d), or (f) there is
no cap/curl in $c_{\mu}$ with rightmost vertex $i$, and so in these
cases we have that
$$D(\lambda^+,\mu)=D(\lambda^-,\mu).$$ 
For $\mu$ as in Figure \ref{nonplus}(b) or (e) there might be a
cap/curl with rightmost vertex $i$.

If $i$ is the second non-trivial vertex in $\mu$ (or $\lambda^+$,
$\lambda^-$), then $l_i(\lambda^-,\mu)$ is even and so
$l_i(\lambda^+,\mu)$ is odd. Also the cap/curl in $\mu$ involved the
first non-trivial vertex in $\mu$ and so its value must be even. Hence
we again have that 
$$D(\lambda^+,\mu)=D(\lambda^-,\mu).$$ 

\begin{figure}[ht]
\includegraphics{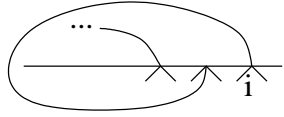}
\caption{The final configuration of $\mu$}
\label{newcurl}
\end{figure}

If $i$ is not the second non-trivial vertex then we must have a
configuration of the form in Figure \ref{newcurl}. Note that
$$l_{i-2}(\lambda^+,\mu)\leq
l_i(\lambda^+,\mu)-1=l_i(\lambda^-,\mu)$$
and as the values are non-increasing in nested chambers we again have
that
$$D(\lambda^+,\mu)=D(\lambda^-,\mu).$$

Thus in all cases we have
\begin{eqnarray*}
p_{\lambda^+\mu}(q)&=&q^{l(\lambda^+,\mu)}\sum_{c\in
  D(\lambda^+,\mu)}q^{-2|c|}\\
&=&qq^{l(\lambda^-,\mu)}\sum_{c\in
  D(\lambda^-,\mu)}q^{-2|c|}\\
&=& qp_{\lambda^-\mu}(q).
\end{eqnarray*}
\end{proof}

\section{Projective resolutions of standard modules}\label{proj}

We now have the combinatorial framework needed to describe projective
resolutions of standard modules for the walled Brauer algebra. This is
inspired by the corresponding result for Khovanov's diagram  algebra
in \cite[Theorem 5.3]{bs2}

\begin{thm}\label{projres}
For each $\lambda \in\Lambda_{r,s}$ there is an exact sequence
$$\cdots\too P_{(a)}^m(\lambda)\too\cdots\too P_{(a)}^1(\lambda)\too
P_{(a)}^0(\lambda)
\too \Delta_{(a)}(\lambda)\too 0$$
where 
$$P_{(a)}^i(\lambda)
=\bigoplus_{\mu\in\Lambda_{(a)}}p_{\lambda\mu}^{(i)}P_{(a)}(\mu).$$
\end{thm}
\begin{proof}
Let $\lambda\in\Lambda_{(a)}$. If $\lambda$ is minimal then
$$\Delta_{(a)}(\lambda)=P_{(a)}(\lambda)=P^0_{(a)}(\lambda)$$ and
$P^m_{(a)}(\lambda)=0$ for all $m\geq 0$ and for all $(a)$ with
$\lambda\in\Lambda_{(a)}$. Thus we may assume that $\lambda$ is not
minimal.

As in Section \ref{repthry} we choose a cap or a curl in $\lambda$ not
containing any smaller caps or curls. We have eight cases to consider
as shown in Figures \ref{capcase} and \ref{curlcase}.  We proceed by
induction on $\deg(\lambda)$. Note that in all cases we have
$\deg(\lambda')<\deg(\lambda)$ and in cases (iii) and (viii)(b) we
also have $\deg(\lambda^-)<\deg(\lambda)$. So we can assume that the
result holds for $\lambda'$ and $\lambda^-$.

In cases (i), (ii), (iv-vii) and (viii)(a) we have by induction a
projective resolution of $\Delta_{(a+1)}(\lambda')$ of the form
$$\cdots\too P_{(a+1)}^m(\lambda')\too\cdots\too
P_{(a+1)}^1(\lambda')\too P_{(a+1)}^0(\lambda') \too
\Delta_{(a+1)}(\lambda')\too 0.$$ In these cases we saw that $\lambda$
and $\lambda'$ are translation equivalent. Applying the exact functor
$\res_{(a+1)}^{\lambda}$ to this resolution and using Theorem
\ref{equiv}(iii) and Proposition \ref{klrel}(i) and (ii) we get a
projective resolution 
$$\cdots\too P_{(a)}^m(\lambda)\too\cdots\too P_{(a)}^1(\lambda)\too
P_{(a)}^0(\lambda)
\too \Delta_{(a)}(\lambda)\too 0$$
as required.

For the cases (iii) and (viii)(b) we set $\lambda^+=\lambda$. By
induction we have projective resolutions of $\Delta_{(a+1)}(\lambda')$ and
$\Delta_{(a)}(\lambda^-)$ of the form
\begin{equation}\label{bot}
\cdots\too P_{(a+1)}^m(\lambda')\too\cdots\too P_{(a+1)}^1(\lambda')\too
P_{(a+1)}^0(\lambda')
\too \Delta_{(a+1)}(\lambda')\too 0
\end{equation}
and
\begin{equation}\label{top}
\cdots\too P_{(a)}^m(\lambda^-)\too\cdots\too P_{(a)}^1(\lambda^-)\too
P_{(a)}^0(\lambda^-)
\too \Delta_{(a)}(\lambda^-)\too 0.
\end{equation}
We also have an exact sequence
$$0\too \Delta_{(a)}(\lambda^-)\stackrel{f}{\too}
\res_{(a+1)}^{\lambda}\Delta_{(a+1)}(\lambda')\too
\Delta_{(a)}(\lambda^+)\too 0.$$

Applying $\res_{(a+1)}^{\lambda}$ to (\ref{bot}) and extending $f$ to
a chain map using (\ref{top}) we get a commutative diagram with exact
rows
$$\begin{array}{ccccccc} \too& P_{(a)}^m(\lambda^-)&\too\cdots\too&
  P_{(a)}^0(\lambda^-)& \too& \Delta_{(a)}(\lambda^-)&\too
  0\\ &\downarrow&&\downarrow&&\downarrow f&\\ 
\too&
  \res_{(a+1)}^{\lambda}P_{(a+1)}^m(\lambda')&\too\cdots\too&
  \res_{(a+1)}^{\lambda}P_{(a+1)}^0(\lambda')& \too&
  \res_{(a+1)}^{\lambda}\Delta_{(a+1)}(\lambda')&\too 0\end{array}$$
which we extend into a double complex by adding $0$s in all
remaining rows.

Taking the total complex of this double complex gives an exact
sequence
\begin{multline}\label{total}
\cdots\too 
P_{(a)}^m(\lambda^-)\oplus\res_{(a+1)}^{\lambda}P_{(a+1)}^{m+1}(\lambda')
\too\cdots\\
\cdots\too
\Delta_{(a)}(\lambda^-)\oplus\res_{(a+1)}^{\lambda}P_{(a+1)}^{0}(\lambda')
\too \res_{(a+1)}^{\lambda}\Delta_{(a+1)}(\lambda')\too 0.
\end{multline}
By Proposition \ref{indres} there is an obvious injective chain map
from
$$\cdots \too 0\too \cdots \too 0\too \Delta_{(a)}(\lambda^-)
\too \Delta_{(a)}(\lambda^-)\too 0$$
to the complex in (\ref{total}), and the quotient gives an exact
sequence
\begin{equation}\label{quot}
\cdots\too 
P_{(a)}^m(\lambda^-)\oplus\res_{(a+1)}^{\lambda}P_{(a+1)}^{m+1}(\lambda')
\too
\cdots\too
\res_{(a+1)}^{\lambda}P_{(a+1)}^{0}(\lambda')
\too \Delta_{(a)}(\lambda^+)\too 0.\end{equation}

By Propositions \ref{lc} and \ref{klrel}(iii) we have
$$\res_{(a+1)}^{\lambda}P_{(a+1)}^0(\lambda')=
\res_{(a+1)}^{\lambda}P_{(a+1)}(\lambda')=
P_{(a)}(\lambda^+)=P_{(a)}^0(\lambda^+).$$
For $m>0$ we have by Proposition \ref{lc} and Proposition \ref{klrel} that
\begin{align*}
P_{(a)}^m(\lambda^-)\oplus\res_{(a+1)}^{\lambda}P_{(a+1)}^{m+1}(\lambda')
&=\bigoplus_{\mu\in\calB(\lambda)}
p_{\lambda^-\mu}^{(m)}P_{(a)}(\mu)\oplus
\bigoplus_{\mu'\in\calB(\lambda')}
p_{\lambda'\mu'}^{(m+1)}\res_{(a+1)}^{\lambda}P_{(a+1)}(\mu')\\
&=\bigoplus_{\mu\in\calB(\lambda)}
p_{\lambda^-\mu}^{(m)}P_{(a)}(\mu)\oplus
\bigoplus_{\mu'\in\calB(\lambda')}
p_{\lambda'\mu'}^{(m+1)}P_{(a)}(\mu^+)\\
&=\bigoplus_{\mu^+\in\calB(\lambda)}\left(p_{\lambda^-\mu^+}^{(m)}+
p_{\lambda'\mu'}^{(m+1)}\right)P_{(a)}(\mu^+)\oplus
\bigoplus_{\mu\in\calB(\lambda),\mu\neq\mu^+}p_{\lambda^-\mu}^{(m)}P_{(a)}(\mu)\\
&=\bigoplus_{\mu^+\in\calB(\lambda)}
p_{\lambda^+\mu^+}^{(m+1)}P_{(a)}(\mu^+)\oplus
\bigoplus_{\mu\in\calB(\lambda), \mu\neq\mu^+}p_{\lambda^+\mu}^{(m+1)}P_{(a)}(\mu)\\
&=\bigoplus_{\mu\in\calB(\lambda)}
p_{\lambda^+\mu}^{(m+1)}P_{(a)}(\mu)
=P_{(a)}^{(m+1)}(\lambda^+)=P_{(a)}^{(m+1)}(\lambda).
\end{align*}
Substituting into (\ref{quot}) we obtain the desired projective
resolution of $\Delta_{(a)}(\lambda)$.
\end{proof}
 
For fixed $(a)$ we can consider the matrices formed by the
$p_{\lambda\mu}(q)$ and the $d_{\lambda\mu}(q)$ with rows and columns
indexed respectively by $\lambda$ and $\mu$ in $\Lambda_{(a)}$. The
next pair of Corollaries follow from the last Proposition in exactly
the same way as in \cite[Corollaries 5.4 and 5.5]{bs2}.

\begin{cor}\label{pord}
The matrix $(p_{\lambda\mu}(-q))$ is the inverse of the matrix 
$(d_{\lambda\mu}(q))$.
\end{cor}

\begin{cor}
We have
$$p_{\lambda\mu}(q)=\sum_{i\geq 0}q^i\dim\Ext^i(\Delta(\lambda),L(\mu)).$$
\end{cor}

\begin{rem}
We have seen that the walled Brauer algebras have the same
combinatoric for decomposition numbers and for projective resolutions
of standard modules as the generalised Khovanov diagram algebras
studied by Brundan and Stroppel \cite{bs1,bs2,bs3,bs4}. They have
shown that these Khovanov algebras are Morita equivalent (in a
limiting sense) to blocks of the general linear supergroup, and that
their quasihereditary covers are Morita equivalent to certain
parabolic category ${\mathcal O}$s. It would be very interesting (if
true) to determine an analogous relationship between these algebras
and the walled Brauer algebra, and to find analogous correspondences
for the Brauer algebra.
\end{rem}

\appendix
\section{Kazhdan-Lusztig polynomials}\label{app}

In this section we shall review the constructions of Kazhdan-Lusztig
polynomials corresponding to $A_{r}\times A_s$ inside $A_{r+s+1}$ and
$A_{n-1}$ inside $D_n$ given respectively by Lascoux and
Sch\"utzenberger \cite{lsgrass} and by Boe \cite{boe}, and how these
can be identified (up to a power of $q$) with the polynomials
associated to valued cap diagrams and valued curl diagrams. In the
former case this was already observed in \cite{bs2}.

We begin by outlining the construction of Boe \cite{boe}. Fixing $W$
of type $D_n$ and a fixed subCoxeter system of type $A_{n-1}$ defines a
dominant set of elements in $W$. These can be identified with words
of the form
$$w=w_n\ldots w_1$$ where each $w_i\in\{\alpha,\beta\}$, such that the
number of $\alpha$s is even. Because of this parity condition the
final element $w_1$ is redundant and is omitted.

Given a partition $\lambda$ we will identify the weight $x_{\lambda}$
with a word $w$ of the above form in the following manner. Fix $m>>0$
so that $m$ is the rightmost vertex in $x_{\lambda}$ lying on a cap or
curl in $c_{\lambda}$, and let $n$ be the number of vertices
labelled $\vee$ or $\wedge$ between $0$ and $m$ inclusive, and we
associate $\lambda$ to the word $w$ obtained by setting $w_i=\alpha$
(respectively $\beta$) if the $(n-i)$th such vertex from the left is
$\vee$ (respectively $\wedge$). We will refer to these vertices as the
\emph{non-trivial} vertices in $x_{\lambda}$.

Note that the identification letters in $w$ read from left to right
correspond to vertices in $x_{\lambda}$ read from right to left.

Lascoux-Sch\"utzenberger introduced the cyclic monoid $Z$ in the
letters $\alpha$ and $\beta$ \cite[Section 4]{lsgrass}. Rather than
repeating their definition, we note that if $w=w'zw''$ with
$z\in Z$ then $z$ corresponds to a line segment in $x_{\lambda}$ where
the non-trivial vertices form a sequence of (possibly nested) caps. If
$w=w'\alpha z\beta w''$ then Boe calls $\alpha$ and $\beta$ a
\emph{linked $\alpha\beta$ pair}; this corresponds to a cap in our
terminology. If
$$w=w'\alpha z_{2r}\alpha z_{2r-1}\alpha\ldots \alpha z_1\alpha z_0$$
with $z_i\in Z$ then Boe calls the rightmost $\alpha$ \emph{terminal}
and each pair of $\alpha$s separated by some $z_{2i}$ a \emph{linked
  $\alpha\alpha$ pair}. Under our correspondence linked $\alpha\alpha$
pairs correspond to curls. As Boe omits $w_1$ but $x_{\lambda}$
retains the corresponding point, a terminal $\alpha$ corresponds to
either a cap or a curl involving the leftmost non-trivial vertex.

Boe next defines a rooted directed tree associated to the word $w$. It
is routine to verify that  this corresponds to the tree with
vertices labelled by the chambers for $x_{\lambda}$, where an edge
connects chamber $A$ to chamber $B$ if chamber $A$ is adjacent to and
surrounds chamber $B$, and the unbounded chambers (separated by
infinite rays) are regarded as a single unbounded chamber via the
point at infinity. 

Thus the root of the tree corresponds to the unique
unbounded chamber, while the terminal nodes correspond to the small
chambers. Certain edges in the tree are marked with a plus sign; these
correspond to edges which cross either a curl or a cap involving the
left-most non-trivial vertex.

Certain pairs of edges in the tree are related by a dotted arrow. We
will describe the diagram version; the equivalence of the two is a
straightforward exercise. Suppose we have a chamber $A$ (possibly
unbounded) inside which there are a series of maximal chambers
$A_1,\ldots ,A_t$ from left to right (possibly containing other
chambers inside them) not separated by the end of a curl. If
the leftmost chamber $A_1$ is formed either by a curl or by a
cap involving the leftmost non-trivial vertex, then there is a dotted
arrow from the edge defined by $A_i$ in $A$ to the edge defined by
$A_{i+1}$ in $A$ for $1\leq i\leq t-1$.

In fact the dotted arrows are redundant in the diagram case: the leftmost chamber in a
curl must always be formed either by a curl or by a cap involving the
left-most non-trivial vertex, and the same is true in any unbounded
chamber with no ray to its left. Chambers formed by caps or with a ray
to their left cannot contain curls or the left-most non-trivial
vertex. Thus we can omit the dotted arrows in our diagrams without any
ambiguity.

Instead of labelling edges with plus signs, we will label chambers by
moving any labels to the vertices at the bottom of their respective
edges.

\begin{example}
An example of the correspondence between curl diagrams and
labelled graphs is given in Figure \ref{graphdiag}. Here we have
included the dotted arrows to emphasise where they occur. Note that the
graph must be reflected in the vertical axis under the correspondence
with the construction for Boe in terms of words in $\alpha$ and $\beta$.
\end{example}

\begin{figure}[ht]
\includegraphics{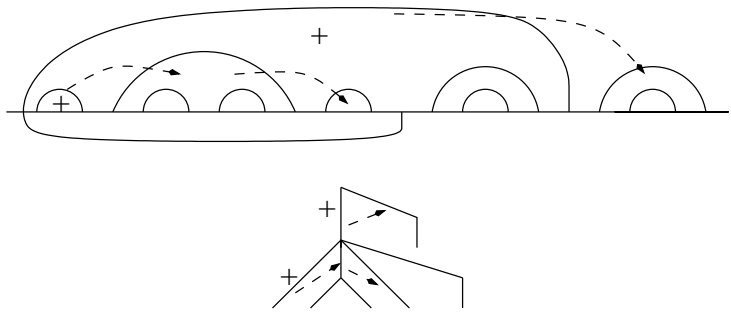}
\caption{The diagram graph correspondence}
\label{graphdiag}
\end{figure}

\begin{rem}
Our construction appears to depend on the choice of $m$ defined by the
rightmost vertex on a cap or curl. However, Boe's construction
(in our diagrammatic form) is not affected by the addition of
arbitrarily many rays to the right. Thus we can carry out all
calculations involving our diagrams in the unbounded setting.
\end{rem}

Boe next associates to pairs of words $(w,y)$ a labelling of the tree
for $w$. Under our identifications this corresponds to a valued curl
diagram. The polynomial $Q_{y,w}(q)$ defined by Boe by summing over
possible labellings corresponds almost exactly to our
$p_{\lambda\mu}(q)$. More precisely, if we denote by $w(\lambda)$ and
$w(\mu)$ the words in $\alpha$ and $\beta$ corresponding to $\lambda$
and $\mu$ (as described at the beginning of this section), then we
have that
$$p_{\lambda\mu}(q)=q^{l(\lambda,\mu)}Q_{w(\lambda),w(\mu)}(q^{-2}).$$

We have considered the relation between Boe's rooted tree construction
and curl diagrams. There is an entirely analogous relation between
the rooted tree construction of Lascoux-Sch\"utzenberger and cap
diagrams. In that case there are no linked $\alpha\alpha$ pairs or
terminal $\alpha$s marked with a plus sign, and thus no chambers
contain chains. The remainder of the construction goes through unchanged.

\newcommand{\etalchar}[1]{$^{#1}$}
\providecommand{\bysame}{\leavevmode\hbox to3em{\hrulefill}\thinspace}
\providecommand{\MR}{\relax\ifhmode\unskip\space\fi MR }
\providecommand{\MRhref}[2]{%
  \href{http://www.ams.org/mathscinet-getitem?mr=#1}{#2}
}
\providecommand{\href}[2]{#2}

\end{document}